\documentclass[10pt,a4paperwide,leqno]{amsart}
\usepackage{amsmath}
\usepackage{amssymb}
\usepackage{amsfonts}
\usepackage{amsthm}
\usepackage{mathrsfs}
\usepackage{dsfont}
\usepackage{mathtools}
\usepackage{bm}

\usepackage[backref=page]{hyperref}
\usepackage{color}

\definecolor{darkgreen}{rgb}{0.5,0.25,0}
\definecolor{darkblue}{rgb}{0,0,1}
\definecolor{answerblue}{rgb}{0,0,0.75}

\hypersetup{colorlinks,breaklinks,
linkcolor=darkblue,urlcolor=darkblue,
anchorcolor=darkblue,citecolor=darkblue}

\newcommand*{\mailto}[1]{\href{mailto:#1}{\nolinkurl{#1}}}

\theoremstyle{plain}
\newtheorem{theorem}{Theorem}
\newtheorem{lemma}[theorem]{Lemma}
\newtheorem{corollary}[theorem]{Corollary}

\theoremstyle{definition}

\theoremstyle{remark}
\newtheorem{remark}[theorem]{Remark}

\numberwithin{theorem}{section}
\numberwithin{equation}{section}

\DeclareMathOperator{\tr}{tr}

\def\mff{{\mathfrak f}}
\def\mx{{\bf x}}
\def\my{{\bf y}}
\def\mk{{\bf k}}
\def\Div{{\rm div}}

\def\pa{\partial}
\def\cal{\mathcal}
\let\mib=\boldsymbol

\def\C{\Bbb{C}}
\def\R{\Bbb{R}}

\def\N{\Bbb{N}}
\def\S{\Bbb{S}}

\def\dscon{\relbar\joinrel\rightharpoonup}
\def\eps{\varepsilon}
\def\malpha{{\mib \alpha}}
\def\mxi{{\mib \xi}}
\def\mlambda{{\mib \lambda}}
\def\meta{{\mib \eta}}
\def\mzeta{{\mib \zeta}}
\def\ph{\varphi}

\newcommand{\cR}{\mathcal{R}}
\newcommand{\cO}{\mathcal{O}}
\newcommand{\cD}{\mathcal{D}}
\newcommand{\cDp}{\mathcal{D}^\prime}
\newcommand{\cC}{\mathcal{C}}
\newcommand{\cS}{\mathcal{S}}
\newcommand{\cA}{\mathcal{A}}
\newcommand{\cB}{\mathcal{B}}
\newcommand{\cF}{\mathcal{F}}
\newcommand{\cI}{\mathcal{I}}
\newcommand{\norm}[1]{\left \lVert#1 \right\rVert}
\newcommand{\cM}{\mathcal{M}}

\newcommand{\cT}{\mathcal{T}}
\newcommand{\abs}[1]{\left\lvert#1\right\rvert }
\newcommand{\bS}{\Bbb{S}}
\newcommand{\weak}{\rightharpoonup}

\newcommand{\inn}[1]{\left\langle#1\right\rangle}
\newcommand{\innb}[1]{\bigl\langle#1\bigr\rangle}
\newcommand{\weakstar}{\xrightharpoonup{\star}}
\newcommand{\action}[2]{\left\langle #1, #2 \right\rangle}
\newcommand{\actionb}[2]{\bigl\langle #1, #2 \bigr\rangle}
\newcommand{\loc}{\operatorname{loc}}
\newcommand{\supp}{\operatorname{supp}}

\newcommand{\ton}{\xrightarrow{n\uparrow \infty}}
\newcommand{\toell}{\xrightarrow{\ell\uparrow \infty}}
\newcommand{\todelta}{\xrightarrow{\delta\downarrow 0}}

\newcommand{\toN}{\xrightarrow{N\uparrow \infty}}

\newcommand{\weakn}{\xrightharpoonup{n\uparrow \infty}}
\newcommand{\seq}[1]{\left\{#1\right\}}

\newcommand{\sign}{\operatorname{sign}}
\newcommand{\meas}{\operatorname{meas}}
\newcommand{\En}{\mathbf{1}}

\makeatletter
\@namedef{subjclassname@2020}{%
  \textup{2020} Mathematics Subject Classification}
\makeatother

\begin{document}

\title[Deterministic and stochastic velocity averaging]
{Velocity averaging under minimal conditions \\ for
deterministic and stochastic kinetic equations with irregular drift}

\author[Erceg]{Marko Erceg}
\address[Marko Erceg]
{Department of Mathematics, 
Faculty of Science, 
University of Zagreb, 
Bijeni\v cka cesta 30, 
10000 Zagreb, Croatia}
\email{\mailto{maerceg@math.hr}}

\author[Karlsen]{Kenneth H. Karlsen}
\address[Kenneth H. Karlsen]
{Department of Mathematics, 
University of Oslo, 
P.O. Box 1053 Blindern, 
NO-0316 Oslo, Norway}
\email{\mailto{kennethk@math.uio.no}}

\author[Mitrovi\'{c}]{D. Mitrovi\'{c}}
\address[Darko Mitrovi\'{c}]
{University of Montenegro, 
Faculty of Mathematics, 
Cetinjski put bb,
81000 Podgorica, 
Montenegro}
\address{University of Vienna, 
Faculty of Mathematics, 
Oscar-Morgenstern-Platz 1, 
1090 Vienna, Austria}
\email{\mailto{darkom@ucg.ac.me}}



\begin{abstract}
This study investigates the $L^1_{\operatorname{loc}}$ 
compactness of velocity averages 
of sequences of solutions $\{u_n\}$ for a class 
of kinetic equations. The equations are examined within 
both deterministic and stochastic heterogeneous environments. 
The primary objective is to deduce velocity averaging results 
under conditions on $u_n$ and the drift 
${\mathfrak f}={\mathfrak f}(t,{\boldsymbol x},{\boldsymbol \lambda})$ 
that are more lenient than those stipulated in previous studies.
The main outcome permits the inclusion of highly irregular drift vectors 
${\mathfrak f} \in L^q$ that adhere to a general non-degeneracy condition.  
Moreover, the sequence $\{u_n\}$ is uniformly bounded 
in $L^p$---for an exponent $p$ allowed to be 
strictly smaller than $2$---under the 
requirement $\frac{1}{p} + \frac{1}{q} < 1$. 
Resolving the matter of strong compactness in velocity averages, 
considering these assumptions, has remained 
an open problem for a long time. The cornerstone of our 
work's progress lies in the strategic employment of the broader 
concept of $H$-distributions, moving beyond the traditional 
reliance on $H$-measures. Notably, our study represents one 
of the first significant uses of $H$-distributions in this context.
\end{abstract}

\date{\today}

\maketitle

\setcounter{tocdepth}{1}
\tableofcontents

\section{Introduction}
Kinetic PDEs are used in a variety of areas of 
physics and mathematics, including statistical mechanics and 
in the study of hyperbolic conservation laws via 
the so-called kinetic formulations \cite{Perthame:2002qy}. 
Conservation laws with non-homogenous ($\mx$-discontinuous) 
flux functions arise in many applications. For example, in 
oil reservoirs, the permeability of rock can differ 
greatly in different locations, resulting in a 
discontinuous flux function. Similarly, in traffic 
flow, the density of vehicles can change abruptly 
at bottlenecks, also resulting in a discontinuous 
flux function. In order to account for both 
randomness and abrupt changes in the underlying system, 
the relevant models take the form of stochastic 
conservation laws with $\mx$-discontinuous flux. 
These macroscopic models can be formulated as 
deterministic and stochastic kinetic equations  
with a transport part that is  non-homogenous and $\mx$-dependent 
in a rough (discontinuous) way. 

In the present study, we 
focus on the strong compactness of velocity averages of 
solutions $u_n$ to a class of kinetic equations in both 
deterministic and stochastic settings. The deterministic form 
of the equations is given by
\begin{equation}\label{eq-1}
	\pa_t u_n+\Div_\mx \bigl (\, \mff(t,\mx,\mlambda) u_n\, \bigr)
	=\pa_{\mlambda}^{\malpha} g_n, 
\end{equation}
for $(t,\mx)\in \R_+\times \R^d=:\R^{d+1}_+$, 
$\mlambda \in \R^m$, and $\malpha \in \N^m$ 
with $\abs{\malpha}=N_\mlambda\ge 0$.

The primary objective of this work is to establish 
velocity averaging results under minimal conditions 
on the solutions $u_n$ and the vector field $\mff$. 
These conditions are notably weaker than those imposed in existing 
literature (see upcoming discussion). Specifically, we assume that
\begin{equation}\label{eq:intro-ass1}
	\text{$\seq{u_n}_{n\in \N}$ 
	is bounded in $L^p(\R^{d+1}_+\times \R^m)$ with $p>1$},
\end{equation}
and that the vector field
\begin{equation}\label{eq:intro-ass2}
	\text{$\mff$ belongs to $L^q(\R^{d+1}_+\times \R^m;\R^d)$ 
	with $\frac{1}{p}+\frac{1}{q}<1$}.
\end{equation}
Without loss of generality, we may assume $p < 2$ (and thus $q > 2$), 
as the case $p \geq 2$ is comprehensively addressed 
elsewhere, see, e.g., \cite{Erceg:2023aa,Lazar:2012aa}. 
We emphasize that the drift vector 
$\mff=\mff(t,\mx,\mlambda)$ is allowed 
to exhibit a high degree of irregularity.

The source term on the right-hand side 
of the kinetic equation \eqref{eq-1} can certainly be expressed 
in more general forms, however, the current source covers a 
majority of interesting applications, including 
hyperbolic conservation laws. 
This balance between generality and simplicity helps to 
keep the technicalities of the paper at a manageable level.
We will assume that
\begin{equation}\label{eq:intro-ass3}
	\text{$\seq{g_n}_{n\in \N}$ is precompact 
	in $W^{-1,\alpha}(\R^{d+1}_+\times \R^m)$, 
	for some $\alpha>1$}.
\end{equation} 

The classical averaging lemmas state that by  
averaging the solution $u=u(t,\mx,\mlambda)$ of a 
kinetic PDE in the microscopic variable $\mlambda$, 
the resulting function, denoted by 
$\actionb{u(t,\mx,\cdot_\mlambda)}{\rho}:=\int u(t,\mx,\mlambda)
\rho(\mlambda) \,d\mlambda$, will have 
better (fractional Sobolev/Besov) regularity 
in $t,\mx$ than the original function $u$, for 
a suitable averaging kernel $\rho$. 

Velocity averaging lemmas have been widely 
used to determine the existence of solutions 
to conservation laws and related PDEs  by analyzing the 
precompactness of sequences of solutions $\seq{u_n}$ to 
kinetic equations like \eqref{eq-1} when 
averaged over the velocity variable 
$\mlambda$. These lemmas are typically 
applied to homogeneous equations, 
where the coefficients $\mff=\mff(\mlambda)$ only depend 
on the velocity variable $\mlambda$. 
This is because if the coefficients 
in the kinetic equation do not vary with position and time, 
the Fourier transform can be applied to 
separate the solutions and coefficients, allowing 
for the estimation of solutions through the 
known coefficients. There is a large body 
of literature on averaging lemmas 
for homogeneous kinetic equations and only a few will be mentioned here. 
The averaging lemmas were first introduced independently 
in \cite{Agoshkov:84} and \cite{Golse-etal:88} in the $L^2$ setting. 
They were later extended to general $L^p$, $1< p<\infty$, 
by \cite{Bezard:94} and \cite{DiPerna-etal:91}. 
This was followed by the optimal Besov space regularity proved 
in \cite{DeVorePetrova:01}. The regularity for the $L^p$ 
case is further improved in the one-dimensional case, precisely 
from $\frac{1}{p}$ to $1-\frac{1}{p}$ 
when $p>2$ by \cite{ArsenioMasmoudi:19}. 
The topic of averaging lemmas under different conditions was 
further examined in various studies 
\cite{ArsenioSaint-Raymond:11, 
GolseSaint-Raymond:02,JabinVega:03,JabinVega:04, 
PerthameSouganidis:98,Tadmor:2006vn,Westdickenberg:02} 
(see \cite{Jabin:2009aa} for a partial review).

In the case where the coefficients depend in a rough way on $\mx$, 
as in \eqref{eq-1}, the primary method for deriving velocity averaging 
results is through the use of \textit{defect measures} that 
take into account oscillations. This approach, originally 
studied for sequences satisfying elliptic estimates 
(see \cite[Chapter 4]{Evans:1990wt}), involves analyzing loss 
of compactness in a weakly converging sequence of functions. 
However, as elliptic estimates do not account 
for oscillations commonly found in hyperbolic problems, 
in order to control them, G\'erard \cite{Gerard:91} 
and Tartar \cite{Tartar:1990mq} independently 
introduced the concepts of micro-local 
defect measures (G\'erard) and $H$-measures (Tartar). 
We will mainly use the notion of $H$-measures. 

$H$-measures are types of measures that generalise defect measures 
and depend on space variables $\mx$ and dual variables $\mxi$. 
They are Radon measures on a 
specific bundle of spaces ($\R^d\times\S^{d-1}$) and 
are associated with weakly converging sequences in $L^2$. 
The existence of $H$-measures follows from the 
Schwartz kernel theorem \cite{HormanderI:90}. 
The advantage of $H$-measures is that they capture 
information about the direction of oscillations in a sequence, which 
defect measures (and Young measures) cannot. This makes 
them useful for studying oscillations, a common 
cause of non-compactness. An example is provided by a 
sequence of plain waves of the form $u_n(\mx)=U(\mx)e^{2\pi i 
\frac{\mx\cdot \mk}{\eps_n}}$, where $0\neq U \in L^2_{\loc}(\R^d)$, 
$\mk \neq 0$, and $\eps_n \to 0$. 
This sequence converges to $0$ 
weakly in $L^2_{\loc}(\R^d)$. The associated defect measure 
is $\abs{U(\mx)}^2\,d\mx$, while
the $H$-measure is $\abs{U(\mx)}^2 d\mx \otimes 
\delta_{\frac{\mk}{\abs{\mk}}}(d\mxi)$, so 
the direction of oscillations is 
captured by the latter.

$H$-measures have been widely employed as 
an effective tool for analyzing PDEs and 
have demonstrated successful applications in 
various fields. They have been used to generalize 
compensated compactness to the variable 
coefficients case \cite{Gerard:91,Misur:2015aa,Tartar:1990mq}, 
in control theory applications \cite{DehmanEtal:2014,LazarZuazua:2014}, 
to establish strong convergence results 
\cite{Panov:1994gf,Panov:2010aa} 
and existence of strong traces \cite{Aleksic:2013aa,Panov:2005lr} 
for conservation laws and related equations, to 
derive velocity averaging results for kinetic 
PDEs \cite{Lazar:2012aa,Erceg:2023aa}, and in 
homogenization theory \cite{AntonicBurazinJankov:23}. 
In this study, we will emphasize the application 
of $H$-measures to kinetic PDEs.

Kinetic PDEs and $H$-measures frequently emerge 
in the kinetic formulation of scalar conservation laws 
that feature a $(t,\mx)$-discontinuous 
flux vector $f=f(t,\mathbf{x},\mlambda)$. In this context, 
$\mff$ in \eqref{eq-1} is recognized as $\partial_\mlambda f$. 
Addressing the question of strong convergence 
of approximate solutions to conservation laws with discontinuous flux, 
\cite{Panov:2010aa} considers solutions $u_n$ that 
are uniformly bounded in $L^\infty$ and assumes that 
the flux meets certain conditions like 
$\mff(\cdot,\cdot,\mlambda) \in L^q_{t,\mx}$, $q>2$. 
This assumption is essential for the analysis, as it pertains 
to certain ``parameter-dependent" $H$-measures. 
A step forward was made in 
\cite{Lazar:2012aa,LazarMitrovic:16} 
using a ``vector-valued" $H$-measure framework, where it was 
allowed that $\mff(\cdot,\cdot,\mlambda) \in L^q_{t,\mx}$ for $q>1$. 
In the case $u_n$ is uniformly bounded in $L^p$ with $p\geq 2$, 
the recent work \cite{Erceg:2023aa} attained the most comprehensive 
non-degeneracy condition on $\mff$ (in a 
general second order setting) while 
simultaneously meeting the regularity 
condition $\mff\in L^q$ with $1/p+1/q<1$. 

In the present work, we address and conclusively bridge 
the outstanding gap by establishing the result for 
$ u_n \in_b L^p$ with $p < 2$, where the 
symbol ``$\in_b$'' signifies uniform boundedness within 
the specified space, and similar notations apply to 
other functional spaces. 

We note that the issue of velocity averaging 
within the context of first-order 
transport equations, encompassing 
both the $L^p$ framework for $p<2$ and a drift that 
is heterogeneously dependent on the spatial variable, has 
remained an open question for quite some time.  
G\'erard's seminal work \cite{Gerard:91} from 1991 
laid down the foundation for the use of microlocal defect measures 
in velocity averaging problems within the $L^2$ 
framework. G\'erard \cite[page 1764]{Gerard:91} 
also raised a question about the potential 
extension of microlocal defect measures 
to address problems in the $L^p$ setting for $p<2$. 
We posit that our main result 
(Theorem \ref{thm:main-result-determ}) 
offers a response to the question 
posed by G\'erard back in 1991, thereby 
bridging a crucial gap in the existing literature.

The key advance of our work lies in a strategic 
use of the broader concept of $H$-distributions. 
In fact, our work represents one of the 
first substantial applications 
of $H$-distributions. Specifically, for any 
$\rho \in C^{N_\mlambda}_c(\R^m)$, we 
will prove that $\actionb{u_n}{\rho}$ converges 
strongly along a subsequence in $L^1_{\loc}(\R^{d+1}_+)$, 
under the assumptions \eqref{eq:intro-ass1}, 
\eqref{eq:intro-ass2}, \eqref{eq:intro-ass3}, and the 
general non-degeneracy condition
\begin{align}\label{non-deg}
	\iint_K \, \sup\limits_{\abs{\mxi}=1} 
	\meas \Bigl\{  \mlambda \in L:\,
	\abs{\xi_0+\actionb{\mff(t,\mx,\mlambda)}{\mxi}} =0
	\Bigr\} \,d\mx\,dt =0, 
\end{align} 
for any $K \times L \Subset \R^{d+1}_+ \times \R^m$. 
Here, $\mxi=(\xi_0,\xi_1,\ldots,\xi_d)=(\xi_0,\mxi')\in \R^{d+1}$ 
and $\meas$ denotes the Lebesgue 
measure with respect to $\mlambda \in \R^m$.  

The $H$-distributions were first introduced in \cite{Antonic:2011aa} 
and then extended in \cite{LazarMitrovic:13,Misur:2015aa}. 
Additionally, we highlight the work \cite{Rindler:15}, which 
establishes a connection between $H$-distributions and Young measures 
via micro-local defect forms, and \cite{AntonicMitrovicPalle:20} where a 
relation between the $H$-distributions and micro-local 
defect forms can be found. The earlier introduced $H$-measure 
objects were limited to solving problems within the $L^2$-framework. 
The $H$-distributions were introduced as a 
versatile extension of $H$-measures, capable of handling sequences 
that are bounded in $L^p$ with $p\neq 2$. In contrast to $H$-measures, 
$H$-distributions are not positive and 
are classified as Schwartz distributions, 
generally not Radon measures in $\mx,\mxi$.  
It turns out that anisotropic 
distributions \cite{AntonicErcegMarin:21} provide 
a more precise description: $H$-distributions 
are distributions of order 0 (measures) in $\mx$ 
and of finite order in $\mxi$. 

\medskip

Another goal of this paper is to use the $H$-distributions to 
establish a velocity averaging result for a class of 
stochastic kinetic equations. These equations are similar in form 
to the deterministic equations, but with the addition of 
an It\^{o} stochastic forcing term $\Phi_n(t,\mx,\mlambda)\, dW_n$. 
Here, $W_n(t)$ is a cylindrical Wiener process \cite{DaPrato:2014aa} 
evolving over a Hilbert space $\Bbb{K}$, and $\Phi_n$ is a noise coefficient 
that belongs to $L^2(\Omega \times [0, T];L_2(\Bbb{K},L^2_{\mx,\mlambda}))$, 
where $T>0$ is fixed. We consider a sequence of stochastic kinetic equations 
(on $[0,T]$) of the form
\begin{equation}\label{eq-s}
	du_n+\Div_\mx \bigl (\, \mff(t,\mx,\mlambda) u_n \, \bigr)\, dt
	=\pa_{\mlambda}^{\malpha} g_n \, dt
	+\Phi_n(t,\mx,\mlambda) \, dW_n(t), 
\end{equation}
indexed over $n\in \N$, with $\abs{\malpha}=N_\mlambda\ge 0$. 
We assume that 
\begin{equation}\label{eq:intro-ass1-stoch}
	\text{$u_n \in_b L^p_{\omega,t,\mx,\mlambda}$ 
	\, and \,
	$u_n\weakn u$ in $L^p_{t,\mx,\mlambda}$, a.s.},
\end{equation}
where $L^p_{\omega,t,\mx,\mlambda}$ is short-hand 
for $L^p(\Omega\times [0,T]\times \R^d\times\R^m)$. 
The drift $\mff$ still satisfies \eqref{eq:intro-ass2} 
and \eqref{non-deg}, while the source 
condition \eqref{eq:intro-ass3} is replaced by 
\begin{equation}\label{eq:intro-ass3-stoch}
	\text{$g_n\in_b L^1_\omega W^{-1,\alpha}_{t,\mx,\mlambda}$ 
	\, and \, 
	$g_n\ton g$ in $W^{-1,\alpha}_{t,\mx,\mlambda}$, a.s., 
	for some $\alpha>1$}. 
\end{equation} 
Finally, we assume that
\begin{equation}\label{eq:intro-assPhi-stoch}
	\begin{split}
		&\text{$\Phi_n \in_b L^2(\Omega \times [0, T];
		L_2(\Bbb{K},L^2_{\mx,\mlambda}))$, \,
		$\Phi_n$ is progressively measurable}.
	\end{split}
\end{equation}

Research on stochastic velocity averaging 
has been relatively sparse. Previous works such as 
\cite{Gess:2018ab,Gess:2017aa,Lions:2013ab,
Nariyoshi:2020aa,Punshon-Smith:2018aa} have focused on 
developing results for homogeneous equations, where $\mff$ is 
independent of $t$ and $\mx$. However, the recent 
work \cite{Karlsen:2022aa} has established a velocity averaging result 
for non-homogeneous kinetic equations, like the one 
shown in \eqref{eq-s}. This result is based on intricate arguments 
involving a specific quasi-Polish space and a.s.~representations 
of random variables in non-metric spaces \cite{Jakubowski:1997aa}. 
These arguments are used to replace the assumption of 
weak $L^2$ convergence of $u_n$ in all variables $(\omega,t,\mx,\mlambda)$ 
by weak convergence in $\mx,\mlambda$, strong convergence in $t$, 
and a.s.~convergence in the probability variable $\omega$. 
The authors then developed a ``vector-valued" $H$-measure framework to 
prove the strong $L^2_{\loc}$ convergence along a subsequence of the velocity 
averages $\actionb{u_n}{\rho}$, assuming 
that $u_n\to 0$ almost surely in the quasi-Polish 
space $L^2_t\bigl((L^2_{\mx,\mlambda})_w\bigr)$, where 
$(L^2_{\mx,\mlambda})_w$ refers to the $L^2$ space with 
a weak topology in the variables $\mx$ and $\mlambda$. 
It is important to note here that the topology in $t$ is strong.

In this study, we introduce a convergence framework 
built upon ``stochastic" $H$-distributions. 
This significantly refines the constraints 
placed on the solutions $u_n$. We operate under the 
general assumption given by \eqref{eq:intro-ass1-stoch}. 
Contrastingly, \cite{Karlsen:2022aa} posited that the sequence 
$\seq{u_n}$ comprised solutions $u_n$ uniformly bounded 
in $L^\infty_{\omega,t,\mx,\mlambda}$. By leveraging a tailored 
$H$-distribution approach, which sidesteps several complexities found in 
\cite{Karlsen:2022aa}, we adopt a boundedness 
condition in $L^p$, for a potentially small $p>1$. 
This expansion sets the stage for a broader range 
of applications in the analysis of kinetic models with 
strongly heterogeneous coefficients. 

Indeed, the $L^p$ framework is natural 
for Boltzmann, Vlasov, and other related kinetic equations, 
since their global solutions are controlled mainly 
through mass, entropy, and moment bounds. 
Once heterogeneous or non-smooth transport 
fields $\mff(\mx,\mlambda)$ are admitted, classical $L^2$ 
microlocal defect measure methods become ill-suited, 
whereas our $L^p$ velocity averaging result remains robust in 
this low-regularity regime; as an example, we 
mention Vlasov-type equations where the coefficients may depend 
on space and velocity in a merely $BV$ (bounded variation)
way \cite{Bouchut:2001zr} (see also \cite{Le-Bris:2004aa,Lerner:2004aa}), 
and we refer to \cite{Karlsen:2022aa} for further
examples of discontinuous drifts arising from nonlinear (S)PDEs. 
Building on the perspectives of DiPerna, Lions, 
and Meyer \cite{DiPerna-etal:91} 
and Tadmor and Tao \cite{Tadmor:2006vn}---where 
\textit{homogeneous} averaging already 
operates with $u\in L^p$ for $p>1$---our results extend 
velocity averaging compactness 
to settings with strong heterogeneity, 
non-flat geometries, and stochastic forcing.
This provides a new tool for tackling open problems
in kinetic theory and nonlinear (S)PDEs where 
heterogeneity and randomness play a decisive role 
(see Section \ref{sec:applic} for some potential applications). 

\medskip

To wrap up this introduction, we want to touch upon 
a broader category of kinetic equations, specifically those that 
involve potentially degenerate diffusive terms. Over the years, there 
have been some studies using $H$-measure techniques to establish 
strong compactness of entropy solutions to mixed hyperbolic-parabolic 
PDEs with heterogeneous flux ($u_n\in_b L^p$ with $p$ large).  
Some examples include the works \cite{Erceg:2023aa,Holden:2009nx}, 
see also the reference list of \cite{Erceg:2023aa}. We refer to 
\cite{Antonic:2008gf,Panov:2011ab} for applications 
of $H$-measures to parabolic and ultra-parabolic equations. 
For mixed hyperbolic-parabolic PDEs with homogeneous 
flux and diffusion coefficients, see \cite{Gess:2018ab,Lions:1994qy,Tadmor:2006vn}. 
We hope to build upon the insights of this paper and the aforementioned 
references to provide strong compactness results under more 
general conditions, particularly for stochastic kinetic equations 
with second-order degenerate terms. 
This topic, along with some of the applications discussed 
in Section \ref{sec:applic}, will be explored in future research.

\medskip

The remaining part of the paper is structured as follows. 
We use Section \ref{sec:applic} to discuss several prospective 
applications of our results, in order to motivate the need for 
a general heterogeneous $L^p$ framework. 
The main theoretical developments are then presented 
in Section \ref{sec:deterministic}, where we analyse deterministic 
equations, and in Section \ref{sec:stochastic}, which focuses on 
stochastic equations. 
Finally, in the appendix (Section \ref{sec:appendix-H-distribution}) 
we provide background material on $H$-measures 
and $H$-distributions, which is essential for 
a full understanding of the results in the main text.

\section{Prospective applications and discussion}\label{sec:applic}
Our heterogeneous $L^p$ averaging theory 
applies to a far broader class of equations than before, 
achieving a level of generality well beyond that of 
homogeneous or $L^2$-based 
heterogeneous theories. To indicate the scope 
and potential impact of this framework, we now discuss 
several prospective applications.
In line with the cited literature, we will here write the 
independent variables as $(t,x,v)$, where $v$ denotes 
the transport (velocity) direction and $f=f(t,x,v)$ the 
particle density (unknown solution); 
this differs from our general notation 
$(t,\mx,\mlambda)$ for the unknown $u(t,\mx,\mlambda)$ 
with drift $\mff(t,\mx,\mlambda)$, and in this discussion the 
drift will be denoted by $a(x,v)$.

\subsection{Radiative-transfer models}
One natural application is radiative-transfer models 
based on linear Boltzmann-type equations with heterogeneous drift. 
In radiative transfer (see, e.g., \cite{Bal:2006aa,Klibanov:2023aa}),
such equations describe the evolution of photon or neutron
densities in inhomogeneous media: the drift (transport) term
$a(x,v)\cdot\nabla_x f$ (or, in conservative form, $\Div_x(a(x,v)f)$) 
and the (linear) collision operator $Q$ together encode spatially varying 
absorption and scattering properties, 
with $a$ sometimes only piecewise smooth  or even discontinuous.  
A representative equation is obtained from \eqref{eq:boltzmann} 
by taking $A,G\equiv 0$ and viewing $Q(f,f)=Q(f)$ 
as a linear collision operator. In this setting, strong 
compactness of averages 
is a key tool for establishing kinetic or macroscopic 
(e.g., diffusion) limits in heterogeneous media;
see, for instance, the classical 
works on transport and diffusion asymptotics 
\cite{BalRyzhik:00,Cercignani:1988aa,DegondGoudonPoupaud:00,
Dumas:2000aa,LevermorePomraning:81,Sentis:80}. 
The heterogeneous averaging principle 
established here provides exactly the compactness 
mechanism needed to pass to the limit in kinetic models 
where the drift or collision cross-section 
depends irregularly on $x$ or even on 
random parameters.

\subsection{Socio-economic models}
A broad range of socio-economic mathematical 
models show that Boltzmann-type kinetic equations 
with strongly heterogeneous coefficients arise naturally 
when agents interact through state-dependent rules 
(set $A,G\equiv 0$ in \eqref{eq:boltzmann} for 
a typical equation). 
Unlike the physical Boltzmann equation---where Newton’s law
fix the transport term in the form $v\cdot\nabla_x$, since 
particle velocities are the derivatives 
of spatial trajectories---these socio-economic models 
prescribe the drift through behavioural or 
strategic rules, which leads to transport operators such as 
$a(x,v)\cdot\nabla_x f$ or $\Div_x(a(x,v)f)$, whose 
$(x,v)$-dependence reflects the heterogeneity of the 
underlying population. In models of opinion formation with 
leadership or political segregation \cite{During:2009aa,During:2015aa} 
this drift and the interaction kernel vary across socio-political 
location. In value-estimation models with irrationality 
and herding \cite{During:2017aa} abrupt changes in 
behaviour across risk zones produce transport fields 
that may even be \textit{discontinuous}. 
Wealth-exchange models 
\cite{Degond:2014ab,Degond:2014aa,During:2007aa} 
generate heterogeneous drifts through coupling 
between wealth and an economic configuration variable. 
Related kinetic opinion models \cite{Toscani:2006aa} 
feature drift and interaction rules depending on the 
evolving distribution of opinions. 

Across all these works, however, the 
genuinely $x$-inhomogeneous Boltzmann equation—with 
the spatial heterogeneity built in at the modelling level—is 
never analyzed directly. Instead, the mathematical theory 
is developed only after passing to suitable 
(formal) limits---grazing-collision, mean-field, or 
hydrodynamic---which lead to more tractable 
Fokker-Planck-type equations. Our main 
velocity averaging theorem gives direct access to these 
heterogeneous Boltzmann formulations and opens 
a route toward global existence theories in 
the spirit of \cite{DiPerna:1989ab}, 
without recourse to such simplified reductions.

\subsection{General relativity}
A more geometric direction involves 
studying Boltzmann and related 
kinetic equations on manifolds $(M,g)$.  
The microlocal structure underlying heterogeneous 
velocity averaging is naturally 
compatible with local coordinate patches 
and low-regularity metric coefficients. 
This makes it feasible to examine general kinetic 
equations on Riemannian manifolds 
\cite{Ben-Artzi:2007aa,Galimberti:2018aa,Karlsen:2022aa} 
and to provide a compactness tool for 
constructing solutions under minimal geometric regularity. 
More generally, we can study Boltzmann 
equations on manifolds that incorporate heterogeneity 
directly in the transport field. 

An important  source of such models 
comes from general relativity \cite{Rendall:2008aa}. 
There, the natural phase space is the set 
of spacetime points $(t,x)$ together with momenta $p$ 
satisfying the relativistic mass constraint, and the 
one-particle distribution function $f$ satisfies a Boltzmann 
equation driven by the Liouville operator associated 
with the geodesic flow and possible external fields. 
In the Einstein-Boltzmann setting of \cite{LeeRendall2013}, 
one can write the Boltzmann equation 
in local coordinates $(x,p)$ as
$$
L_X f=
p^\alpha \partial_{x^\alpha} f
- \Gamma^i_{\alpha\beta}(x)\,p^\alpha p^\beta\,\partial_{p^i} f
+ \cdots,
$$
so the coefficients in the momentum derivatives 
already carry the full $x$-dependence of the metric and connection. 
Here $x=(x^\alpha)_{\alpha=0}^3$ denotes spacetime 
coordinates, with $x^0=t$ the time variable 
and $(x^1,x^2,x^3)$ the spatial variables, and
$\Gamma^i_{\alpha\beta}(x)$ are the Christoffel symbols
of the Levi-Civita connection associated with the metric $g$. 
A key observation in \cite{LeeRendall2013} is that, 
by passing to orthonormal frame variables $v^\mu$, 
the collision operator can be written in exactly the same algebraic form 
as in special relativity; in this sense the curved-space collision operator 
can be ``flattened'' (which is desirable). 
However, this comes at the price that the transport part 
of the equation turns into
$$
v^\mu e^\alpha_\mu(x)\,\partial_{x^\alpha} f
+ \text{(connection terms in $x,v$)}\cdot\nabla_v f,
$$
where $e^\alpha_\mu(x)$ is an $x$-dependent tetrad 
(a local orthonormal frame on $(M,g)$). 
Thus, the transport field even becomes a nonlinear function 
$a(x,v)=e^\alpha_\mu(x)v^\mu$ which varies with $x$ 
through the geometry of $(M,g)$, and this $x$-dependence 
reflects curvature that cannot be 
straightened away unless the spacetime is flat. 

The celebrated global existence theory 
for renormalized solutions 
\`a la DiPerna and Lions~\cite{DiPerna:1989ab}, 
which is by now classical in the flat, spatially 
homogeneous setting with constant transport velocity, 
is open for heterogeneous Boltzmann equations 
on curved backgrounds or rough geometries 
(but see \cite{Dudynski:1992aa,Jiang:2008aa} for 
the relativistic Boltzmann equation). 
These manifold-based models therefore 
provide a natural class of applications 
where heterogeneous velocity 
averaging---allowing fully variable coefficients 
in the transport field---are genuinely needed, and where 
our main result can serve as a key compactness tool in future 
existence and stability theories.  
Similar considerations also apply to the simpler 
Einstein-Vlasov system, where the Vlasov equation 
describes transport along the geodesic flow with
coefficients depending on the spacetime point $(t,x)$ 
through the underlying metric 
(see, for example, \cite{Rein:1992aa}).

We note in passing that, although the 
Einstein-Boltzmann framework show that 
curved backgrounds can produce heterogeneous 
transport fields that resist straightening, not 
every manifold-based kinetic model 
behaves this way. In the geometric linear 
Boltzmann setting of \cite{Han-Kwan:2015aa}, 
the quadratic Hamiltonian permits a change of 
variables leading to a constant transport field, so 
classical homogeneous velocity averaging applies almost directly.

\subsection{Stochastic kinetic models}
Stochastic perturbations have recently become a central theme in
kinetic theory and mathematical fluid mechanics,
where noise is introduced both to model unresolved
turbulent effects and to exploit its regularizing and
stabilizing properties. This direction has been actively
pursued in recent years by many authors
(see, e.g., \cite{Breit:2018aa,Debussche:2010fk,Debussche:2012rt,
fehrman2023nonequilibrium,Flandoli:2008vn,
Flandoli:2010yq,FlandoliLuongo2023} among many others)
in the analysis of stochastic transport, nonlinear SPDEs,
and Navier–Stokes dynamics. 
In several turbulence-inspired models, such
as Kraichnan-type stochastic transport \cite{Kraichnan1968} (see also
\cite{FlandoliLuongo2023,LototskyRozovsky2017}), the
random velocity fields belong only to fractional Sobolev
or H\"older classes with regularity exponent $<1$, so the
associated drift coefficients are far from smooth.

Our stochastic velocity averaging theorem 
fits naturally into this context: it provides a tool 
for kinetic models with noise and remains effective 
in the presence of non-smooth, heterogeneous transport 
and noise coefficients. In particular, it paves the way 
for strong compactness results for stochastic kinetic equations \cite{Debussche:2010fk,Debussche:2012rt,
Karlsen:2022aa,Punshon-Smith:2018aa} and 
for global existence theories for Boltzmann-type equations 
with heterogeneous (possibly discontinuous) 
and stochastic transport, thereby 
extending the theory of \cite{DiPerna:1989ab}.

A representative example is the Boltzmann model of 
\cite{Punshon-Smith:2018aa}, which incorporates 
transport noise in the velocity direction $v$ 
and serves as a prototype for stochastic 
Boltzmann dynamics. The existence of renormalized  
solutions to this equation was established 
in \cite{Punshon-Smith:2018aa} in the homogeneous 
transport regime, under a structural assumption
on the noise-related terms. That work also provides a 
discussion of the physical motivation and relevance 
of the stochastic Boltzmann model, 
together with many pertinent references.
As in the classical setting \cite{DiPerna:1989ab}, 
the existence proof in \cite{Punshon-Smith:2018aa} 
relies in a crucial way on velocity averaging estimates.
Motivated by this, one may apply our stochastic  
averaging result to Boltzmann-type equations 
that incorporate heterogeneous transport 
and stochastic perturbations, of the form
\begin{equation}\label{eq:boltzmann}
	\begin{split}
		d f&+ \Div_x\bigl(a(x,v) f\bigr)\, dt
		= \Div_v \bigl(A(x,v)\nabla_v f\bigr)\, dt
		\\ &\qquad
		+ G(x,v,f)\, dW(t)
		+ Q(f,f)\, dt,
	\end{split}
\end{equation}
where $a(x,v)$ is the (possibly irregular and heterogeneous) 
spatial drift field, $A(x,v)$ is an anisotropic 
velocity-diffusion matrix, $G(x,v,f)\, dW$ encodes external 
random perturbations, $f=f(t,x,v)$ is the particle density, 
and $Q(f,f)$ is the Boltzmann collision operator. 

Here the coefficients $A$ and $G$ may be 
regarded as given and fairly general, 
but in applications one may imposes 
structural conditions such as ``superparabolicity". 
Informally, this means that the noise 
is arranged so that its It\^o correction does 
not dominate the velocity dissipation 
already encoded in the second-order $A$-term; 
in the extremal case of exact cancellation, the net 
effect of the noise leaves the 
underlying equation essentially hyperbolic. 
The model of \cite{Punshon-Smith:2018aa} 
provides a concrete example of this: 
there the transport noise enters in Stratonovich form in 
the velocity direction $v$, and after applying the 
Stratonovich-It\^o conversion rule the equation can 
be cast into the form \eqref{eq:boltzmann} 
with $a(x,v)\equiv v$ and a precise algebraic relation 
between $A$ and $G$. 

The model \eqref{eq:boltzmann} fits 
into the abstract form \eqref{eq-s} by viewing the collision 
and diffusion terms as part of the source. 
Suitable structural and growth assumptions on the 
coefficients $a$, $A$, $G$, and $Q$ are of course 
required to ensure that entropy and moment bounds continue to 
hold in an appropriate form for \eqref{eq:boltzmann}.

\subsection{Homogenization and asymptotic limits} 
Beyond existence theories, the same 
compactness mechanism provides a tool 
for homogenization of heterogeneous Boltzmann-type equations 
\eqref{eq:boltzmann} and related kinetic models 
with rough coefficients in complex geometries. 
A key reference is the work of Golse and collaborators 
\cite{Bernard:2010aa,CagliotiGolse:03,Golse:10}, 
who studied the linear Boltzmann equation 
in periodically perforated domains via an 
extended phase-space formulation. There, the microscopic 
geometry---periodic holes at a critical scaling---induces in the 
limit an additional variable (the time since the last collision) 
and leads to effective damping and memory effects. 
This shows that even in linear settings with smooth coefficients, 
the homogenization of kinetic equations can produce subtle 
macroscopic phenomena that hinge 
on fine compactness (averaging)  properties.

Our results makes it possible to generalize 
this theory to equations with irregular drifts, 
nonlinear collision operators, stochastic perturbations, 
and general manifold geometries. It also connects naturally to 
the theory of multiscale limits and diffusion approximations 
in kinetic theory. Classical diffusion and 
radiative-transfer limits 
\cite{BardosGolsePerthameSentis:88,Bensoussan:1978aa} 
rely on periodic or smooth structures, whereas the present 
heterogeneous velocity averaging 
theorem allows one to treat non-periodic, ergodic, 
or stationary random environments in both deterministic 
and stochastic settings. This opens the way to homogenized 
kinetic and radiative-transfer equations on manifolds 
or random domains with minimal regularity of the drift field, 
and to couplings with degenerate diffusion operators in 
the ultra-parabolic framework of \cite{Panov:2009oq}. 

\subsection{Summary of potential applications} 
All the models mentioned above share a 
common feature: they involve kinetic equations 
with genuinely heterogeneous and potentially irregular
transport fields, possibly coupled to stochastic effects. 
To treat all these models within a unified framework, 
one needs a heterogeneous $L^p$ velocity averaging theory. 
The $L^p$ setting is natural because the 
equations of interest involve transport 
fields $a(x,v)$ that depend on $x$ and may be non-smooth, 
placing them beyond the reach of classical 
microlocal defect measure theory. 

Working instead with $u\in L^p$, for possibly small $p>1$, 
provides exactly the flexibility needed to match the 
integrability structure of renormalized solutions for
Boltzmann- and Vlasov-type equations, and 
heterogeneous $L^p$ averaging continues to yield
compactness even when the drift lacks continuity. 
At the same time, this makes the
heterogeneous theory consistent with the homogeneous 
$L^p$ averaging of DiPerna, Lions, and
Meyer \cite{DiPerna-etal:91} and of Tadmor 
and Tao \cite{Tadmor:2006vn}, where averaging
already operates in the ``$p>1$" regime; in 
terms of compactness, our results can be 
viewed as the first heterogeneous
extension of that classical $L^p$ framework.

\section{Deterministic velocity averaging}\label{sec:deterministic}

This section presents the main result along 
with its proof, which relies on the concept of 
$H$-distributions discussed in the appendix. 
Before stating the theorem and proving it, we outline the 
main ideas of the argument.
Our objective is to establish that if $u_n\weak 0$ 
in $L^p$ with $p>1$, then (along a not relabelled 
subsequence) the velocity averages 
$\bar u_n:=\actionb{u_n}{\rho}\to 0$ in $L^1_{\loc}$. 
To this end, one expresses $\abs{\bar u_n}$ as the product 
$\bar u_n V_n$, where $V_n=\sign(\bar u_n)$. We then pass to a 
subsequence (not relabelled) converging weak--$\star$ 
in $L^\infty$ to a function $V$. Then, we set $v_n:=V_n-V$, 
which converges weak-$\star$ to zero. 
The goal is to prove that the distributional limit 
of the product sequence $\bar u_nv_n$ is zero. If this is the case, then it 
follows immediately that $\int \ph \abs{\bar u_n}\to 0$ for any 
compactly supported continuous function $\ph$, so that $\bar u_n$ 
converges strongly to zero in $L^1_{\loc}$. 

To achieve this, the objective is to characterize the limit 
of more general products of the 
form $\bar u_n \cA_{\psi(\mxi/\abs{\mxi})}(v_n)$, 
where $\cA_\psi$ is a multiplier operator with symbol $\psi$. 
The function $\psi$ can be any $C^{d+1}$ function defined 
on the unit sphere $\S^d\subset \R^{d+1}$, and 
$\mxi=(\mxi_0,\mxi')\in \R^{d+1}$ 
is the dual (Fourier) variable of $(t,\mx)$, where functions 
defined for $t\geq 0$ are extended by zero for $t<0$. 
It is worth noting that $\cA_\psi$ becomes the 
identity operator when $\psi\equiv 1$, so 
one recovers  the original product sequence $\bar u_nv_n$ for a 
special choice of the ``dual" test function $\psi$. 

In essence, the distributional limit of 
$\bar u_n \cA_{\psi(\mxi/\abs{\mxi})}(v_n)$ is 
characterized by an $H$-distribution $B$. This distribution 
can be seen as a weakly-$\star$ measurable mapping 
that belongs to the Bochner space 
$L^r_{w\star}\bigl(\R^{d+1+m};(C^{d+1}(\S^d))'\bigr)$ 
for any $r\in (1,p)$. For details and references, see 
Appendix \ref{sec:appendix-H-distribution}. 
If the solutions $u_n$ exhibit enough uniform 
integrability (i.e., if $p$ is large), the functional 
$B$ becomes measure-valued, 
$B\in L^r_{w\star}\bigl(\R^{d+1+m};\cM(\S^d)\bigr)$, 
which is a simpler case.

To prove the next theorem and the strong convergence of 
the velocity averages of $u_n$, we rely on the localization principle 
for the $H$-distribution $B$, which shares the differential structure 
of the kinetic PDEs satisfied by $u_n$. By carefully applying the 
non-degeneracy condition \eqref{non-deg}, we show 
that $B=0$, which yields the desired conclusion 
and leads to the theorem's proof.

\begin{theorem}\label{thm:main-result-determ}
For each $n\in \N$, let $u_n$ satisfy the kinetic equation \eqref{eq-1}, 
assuming that the conditions \eqref{eq:intro-ass1}, \eqref{eq:intro-ass2},  
and \eqref{eq:intro-ass3} hold. Suppose that the sequence $\seq{u_n}_{n\in \N}$ 
converges weakly to zero in $L^p(\R^{d+1}_+\times \R^m)$,
\begin{equation}\label{eq:hn-weak-Lp}
	u_n \weakn 0 \quad 
	\text{in $L^p(\R^{d+1}_+\times \R^m)$}.
\end{equation}
Assuming the non-degeneracy condition given by \eqref{non-deg}, 
a subsequence (not relabelled) of the sequence $\seq{u_n}$ exists, such 
that for any $\rho \in C^{N_\mlambda}_c(\R^m)$,
\begin{equation}\label{res-1}
	\lim_{n\to \infty}\actionb{u_n}{ \rho} 
	=\lim_{n\to \infty}\int_{\R^m}  
	u_n(t,\mx,\mlambda)
	\rho(\mlambda) \, d\mlambda = 0
	\quad \text{in $L^1_{\loc}(\R^{d+1}_+)$}.
\end{equation}
\end{theorem}

\begin{remark}\label{rem:non-zero-weak-limit}
Without loss of generality, we can assume that $u_n$ converges 
weakly to zero, as given by \eqref{eq:hn-weak-Lp}. 
If this is not the case and $u_n$ weakly converges to $u$ instead, then 
based on the general assumption \eqref{eq:intro-ass3}, we can 
focus on the equation for $\tilde{u}_n := u_n - u.$
This is essentially the same equation, but with $g_n$ 
substituted by a modified source $\tilde{g}_n$. 
The modified source $\tilde{g}_n$ accounts for 
$\pa_t u + \Div_\mx \bigl ( \, \mff(t,\mx,\mlambda) u \, \bigr)$.
\end{remark}

\begin{proof} 
In what follows, it is always understood that the 
functions $u_n=u_n(t,\mx,\mlambda)$ and the later introduced 
$v_n=v_n(t,\mx)$ are defined 
respectively on $\R^{d+1+m}$ and $\R^{d+1}$, by extension 
taking the value zero for $t<0$.

To analyse the $H$-distribution (denoted by $B$), we first derive the 
localization principle from the kinetic equation \eqref{eq-1}, as is customary. 
The localization principle helps describing the support of 
the micro-local defect functional $B$. 
For fixed $\tilde{\rho}\in C^{N_\mlambda}_c(\R^m)$ and 
$\tilde{\ph}\in C^1_c(\R^{d+1})$, let $V$ be the 
${L}^\infty(\R^{d+1})$ weak-$\star$ limit obtained along 
a subsequence (not relabelled) of the 
functions $V_n$ defined as follows:
\begin{equation}\label{vn}
	V_n(t,\mx)=
	\begin{cases}\displaystyle
		\frac{\tilde{\varphi}(t,\mx)
		\overline{\actionb{u_n}{\tilde{\rho}}(t,\mx)}}
		{\abs{\actionb{u_n}{\tilde{\rho}}(t,\mx)}}, 
		& \actionb{u_n}{\tilde{\rho}}(t,\mx) \neq 0
		\\ 
		0, & \mathrm{otherwise}
	\end{cases}.
\end{equation}  
Let $v_n=V_n-V$, and observe that the sequence 
$\seq{v_n}$ satisfies the conditions of Theorems 
\ref{thm:H-distr} and \ref{thm:H-distr-extend}  
on the existence of $H$-distributions, 
that is, $\seq{v_n}$  is a uniformly compactly supported  
sequence $\seq{v_n}\subset L^\infty(\R^{d+1})$ such 
that (along a non-relabelled subsequence) 
$v_n\weakstar 0$ in $L^\infty(\R^{d+1})$ and thus
\begin{equation}\label{eq:vn-weak-conv}
	v_n\weakn 0 
	\quad \text{in $L^{p_v}(\R^{d+1})$ for any $p_v<\infty$}.
\end{equation}
Let us consider a test function of the form
\begin{equation}\label{eq:test-function-phi-n}
	\begin{split}
		& \phi_n(t,\mx,\mlambda)=
		\rho(\mlambda)\tilde{\rho}(\mlambda) 
		\ph(t,\mx)\tilde{\ph}(t,\mx) \overline{\left(\cT_{-1}\circ 
		\cA_{\overline{\psi}\left(\mxi/\abs{\mxi}\right)}(v_n)\right)(t,\mx)},
		\\ & 
		\quad \text{where} \quad 
		\psi \in C^{d+1}(\S^{d}), \quad 
		\ph \in C^1_c(\R^{d+1}), \quad 
		\rho \in C^{N_\mlambda}_c(\R^m).
	\end{split}
\end{equation}
Note that $\cT_{-1}$ denotes the 
Riesz potential operator, i.e., the multiplier operator with the 
symbol $\frac{1}{\abs{\mxi}}$, recalling that $\mxi=(\xi_0,\xi_1,\ldots,\xi_d)
=(\xi_0,\mxi')$. By Lemma \ref{thm:riesz}, we know 
that $\cT_{-1}(v_n)\to 0$ as $n\to \infty$ 
strongly in $L^r_{\loc}(\R_+^{d+1})$ for any $r\in (1,\infty)$. 
Thus, $\seq{\phi_n}$ converges {strongly} to zero in 
$W_{\loc}^{s,r}(\R^{d+1}\times \R^m)$ for any $s<1$ 
and $r\in (1,\infty)$, see also 
the discussion after \eqref{eq:test-func-random}. 
Moreover, the same holds for $\pa_\mlambda^{\malpha} \phi_n$, 
noting that the only dependence on $\mlambda$ 
in $\phi_n$ comes from $\rho$ and $\tilde{\rho}$, 
see \eqref{eq:test-function-phi-n}.

We test the kinetic equation \eqref{eq-1} against the function $\phi_n$ 
and then let $n\to \infty$ along the subsequence defining 
the $H$-distribution $B$ generated by the sequences 
$\seq{\tilde{\rho} \tilde{\varphi} u_n}$ 
and $\seq{v_n}$, see Theorem \ref{thm:H-distr}. 
As the sequence $\seq{g_n}$ is strongly precompact in 
$W^{-1,\alpha}(\R^{d+1}\times \R^m)$ for some $\alpha>1$, 
see \eqref{eq:intro-ass3}, 
and $v_n$ is uniformly bounded in $L^\infty(\R^{d+1})$ (with 
uniform compact support), we can conclude with the following 
localization principle for $B$ (similar details can be 
found in \cite[Theorem 3.1]{LazarMitrovic:13})
\begin{equation}\label{v2} 
	\begin{split}
		&\actionb{\rho \ph \psi 
		\bigl(1,\mff(t,\mx,\mlambda)\bigr)\cdot \mxi}{B}  
		=\actionb{\rho \ph \psi 
		\bigl(\xi_0+ \mff \cdot \mxi'\bigr)}{B} 
		\\  & \,\, 
		= \lim_{n\to \infty}\int_{\R^{d+1+m}}
	 	\rho(\mlambda) \tilde{\rho}(\mlambda) 
		\ph(t,\mx) \tilde{\ph}(t,\mx) u_n(t,\mx,\mlambda) 
		\\ & \quad\qquad\qquad\qquad  \times
		\bigl(1,\mff(t,\mx,\mlambda)\bigr)  
		\cdot \overline{\cA_{\overline{\psi}\bigl(\frac{\mxi}{\abs{\mxi}}\bigr)
		\frac{\mxi}{\abs{\mxi}}}(v_n)}\, d\mlambda\, d\mx \, dt
		\\ &\,\,  
		=\frac{(-1)^{N_\mlambda}}{2\pi i}\lim_{n\to \infty}
	 	\action{\pa^{\malpha}_\mlambda (\rho(\mlambda)\tilde{\rho}(\mlambda)) 
		\ph(t,\mx) \tilde{\ph}(t,\mx) 
		\overline{\cA_{\overline{\psi}\left(\frac{\mxi}{\abs{\mxi}} \right)
		\frac{1}{\abs{\mxi}}}(v_n)(t,\mx)}}{g_n} =0,
	\end{split}
\end{equation} 
exploiting that 
$$
(t,\mx,\mlambda)\mapsto 
\pa^{\malpha}_\mlambda(\rho(\mlambda)\tilde{\rho}(\mlambda)) 
\ph(t,\mx) \tilde{\ph}(t,\mx) 
\overline{\cA_{\overline{\psi}\left(\!\frac{\mxi}{\abs{\mxi}}\right)
\frac{1}{\abs{\mxi}}}(v_n)(t,\mx)}
$$ 
is bounded in the Sobolev space 
$W^{1,\alpha'}(\R^{d+1}\times \R^m)$ and $g_n\to 0$ 
strongly in the dual space $W^{-1,\alpha}(\R^{d+1}\times \R^m)$, 
where $1/\alpha+1/\alpha'=1$. Based on \eqref{eq:intro-ass3}, it 
can be concluded that $g_n$ strongly converges 
to $g$ in $W^{-1,\alpha}(\R^{d+1}\times \R^m)$ for some sub-sequential 
limit $g\in W^{-1,\alpha}(\R^{d+1}\times \R^m)$.  
Without loss of generality, we may assume that $g=0$; otherwise, 
consider $g_n-g$ and use the fact that 
$\cA_{\psi\left(\frac{\mxi}{\abs{\mxi}} \right)\frac{1}{\abs{\mxi}}}(v_n)$ 
weakly converges to zero in $W^{1,\alpha'}_{\loc}(\R^{d+1}\times \R^m)$. 
This implies
\begin{equation}\label{eq:ass-g-zero}
	\lim_{n\to \infty} \action{\pa^{\malpha}_\mlambda
	(\rho(\mlambda) \tilde{\rho}(\mlambda)) 
	\ph(t,\mx) \tilde{\ph}(t,\mx) 
	\overline{\cA_{\overline{\psi}\left(\frac{\mxi}{\abs{\mxi}} \right)
	\frac{1}{\abs{\mxi}}}(v_n)(t,\mx)}}{g} =0.
\end{equation}

In \eqref{v2}, $B$ denotes the $H$-distribution (see 
Theorem \ref{thm:H-distr}) generated by the two 
sequences $\seq{\tilde{\rho}\tilde{\ph} u_n}\subset L^p(\R^{d+1+m})$ and 
$\seq{v_n}\subset L^\infty_c(\R^{d+1})$. More specifically, $B$ is a continuous 
bilinear functional defined on the tensor product space 
$L^{r'}(\mathbb{R}^{d+1+m}) \otimes C^{d+1}(\mathbb{S}^d)$, where $r'$ denotes 
the dual exponent of a number $r\in (1,p)$ such that 
$\frac{1}{p}+\frac{1}{r'}+\frac{1}{p_v}=1$, where $p_v$ is taken 
from \eqref{eq:vn-weak-conv}. In other words, $r'=\frac{p_vp'}{p_v-p'}$. 
Due to the arbitrariness of $p_v$, the exponent $r'$ ($>p'$) 
can be made arbitrary close to $p'$.

Given a Borel set $S \subset \mathbb{R}^D$ ($D\ge 1$) 
and a separable Banach space $\bigl(X, \norm{\cdot}_X\bigr)$, we 
denote the space of $X$-valued $p$-integrable functions on $S$ with respect 
to the Lebesgue measure by $L^p(S;X)$, where $p$ belongs 
to the interval $[1,\infty)$. We will use the Bochner integral 
to define integrability. For general background material 
on Bochner spaces, we refer to \cite[Chapters 1 \& 2]{Hytonen:2016aa}. 
According to Theorem \ref{thm:H-distr-extend}, 
the $H$-distribution $B$ can be viewed as a continuous 
linear functional on the Bochner space $L^{r'}(\R^{d+1+m};C^{d+1}(\S^d))$, 
and thus as an object in $L^r_{w\star}\bigl(\R^{d+1+m};(C^{d+1}(\S^d))'\bigr)$.  
We refer to Appendix \ref{sec:appendix-H-distribution} for further details.

Notationally, the $H$-distribution $B$ acts upon any test function
$\Psi=\Psi(t,\mx,\mlambda,\mxi)$ that belongs to 
$L^{r'}(\R^{d+1+m};C^{d+1}(\S^d))$ 
through the notation $\action{\Psi}{B}$, or occasionally 
denoted as $B(\Psi)$. In view of 
assumption \eqref{eq:intro-ass2}, $\mff\in L^q$ with $q>p'$ and 
thus, by picking $r$ such that $p'<r'\leq q$, 
$\Psi:=\rho(\mlambda) \ph(t,\mx) \psi(\mxi) 
\bigl(\xi_0+ \mff (t,\mx,\mlambda)\cdot \mxi'\bigr)$ 
is such an admissible test function. 

Numerous independent variables and functions dependent on 
these variables are involved. To aid in the understanding of which 
variable each function depends on, we continue to include 
the independent variables while discussing the functions within a 
pairing with the $H$-distribution $B$ (and other pairings), 
which means that sometimes we write 
$\actionb{\Psi(t,\mx,\mlambda,\mxi)}{B}$ 
instead of $\action{\Psi}{B}$.  

Going back to \eqref{v2}, as $\ph$, $\rho$, and $\psi$ 
are arbitrary, we can conclude that
\begin{align}\label{v3} 
	\action{\bigl(\xi_0+ \mff \cdot \mxi'\bigr)\Phi}{B}=0,
	\quad  \Phi=\Phi(t,\mx,\mlambda,\mxi) 
	\in L^{a}(\R^{d+1+m};C^{d+1}(\S^d)),
\end{align} 
where $\frac{1}{a}+\frac{1}{q}=\frac{1}{r'}<\frac{1}{p'}
\Longleftrightarrow a=\frac{qr'}{q-r'}
>\frac{qp'}{q-p'}$. Note that this choice of $a$ guarantees that
$\bigl(\xi_0+ \mff \cdot \mxi'\bigr)\Phi \in 
L^{r'}(\R^{d+1+m};C^{d+1}(\S^d))$, remembering that $\mff \in L^q$ 
($p'<r'\leq q$). Indeed, recalling that the algebraic tensor product 
$L^a(\R^{d+1+m})\otimes C^{d+1}(\S^d)$
is dense in $L^a(\R^{d+1+m};C^{d+1}(\S^d))$, see for example 
\cite[Lemma 1.2.19]{Hytonen:2016aa}, it is enough to verify that 
the equation \eqref{v3} holds for any function 
$\Phi$ from $C_c(\R^{d+1})\otimes C_c(\R^m)
\otimes C^{d+1}(\S^{d})$ or from $C_c(\R^{d+1+m})
\otimes C^{d+1}(\S^{d})$.

In what follows, for a fixed but arbitrary $M>\sqrt{3}$, 
we denote by $\chi_M(t,\mx,\mlambda)$ the characteristic function of the set 
$\seq{(t,\mx,\mlambda):\abs{\mff(t,\mx,\mlambda)}\leq M}$:
\begin{equation*}
	\chi_M(t,\mx,\mlambda)=\En_{\seq{(t,\mx,\mlambda):
	\abs{\mff(t,\mx,\mlambda)}\leq M}}(t,\mx,\mlambda).
\end{equation*}
Then we specify
$$
\Phi=\Phi(t,\mx,\mlambda,\mxi)
=\frac{\chi_M(t,\mx,\mlambda)\varphi(t,\mx) \rho(\mlambda)\psi(\mxi)
\bigl(\xi_0+\mff(t,\mx,\mlambda)\cdot \mxi'\bigr)}
{\abs{\xi_0+\mff(t,\mx,\mlambda) \cdot \mxi'}^2
+\delta \abs{\mxi'}^2},
$$ 
for some arbitrary $\delta>0$ and $\varphi \otimes\rho\otimes \psi
\in C_c(\R^{d+1})\otimes C_c(\R^m) \otimes C^{d+1}(\S^{d})$. 
It is evident that the functions $\Phi$ and 
$\bigl(\xi_0+\mff \cdot \mxi'\bigr)\Phi$ 
both belong to $L_c^\infty(\R^{d+1+m};C^{d+1}(\S^d))$. 

Our objective is to obtain $B\equiv 0$ by 
letting $\delta\to 0$ and using the non-degeneracy 
condition \eqref{non-deg}. Initially, let us observe that
\begin{equation}\label{v4}
	\begin{split}
		0 & =\action{\bigl(\xi_0+\mff \cdot \mxi'\bigr)\Phi}{B} 
		\\ & 
		=\action{\frac{\chi_M(t,\mx,\mlambda)
		\varphi(t,\mx)\rho(\mlambda) 
		\psi(\mxi)\abs{\xi_0+\mff(t,\mx,\mlambda) \cdot \mxi'}^2}
		{\abs{\xi_0+\mff(t,\mx,\mlambda) \cdot \mxi'}^2
		+\delta \abs{\mxi'}^2}}{B}
		\\ & 
		=\actionb{\chi_M(t,\mx,\mlambda)
		\varphi(t,\mx)\rho(\mlambda) \psi(\mxi)}{B}
		\\ & \qquad 
		-\action{\frac{\chi_M(t,\mx,\mlambda)
		\ph(t,\mx)\rho(\mlambda) \psi(\mxi) \delta \abs{\mxi'}^2 }
		{\abs{\xi_0+\mff(t,\mx,\mlambda) \cdot \mxi'}^2
		+\delta \abs{\mxi'}^2}} {B}
		\\ & 
		=\actionb{\chi_M(t,\mx,\mlambda)\varphi(t,\mx)
		\rho(\mlambda) \psi(\mxi)}{B}
		\\ & \qquad
		-\actionb{\chi_M(t,\mx,\mlambda)
		\ph(t,\mx)\rho(\mlambda) \psi(\mxi)
		H_\delta(t,\mx,\mlambda,\mxi)} {B},
	\end{split}
\end{equation} 
where
$$
H_\delta=H_\delta(t,\mx,\mlambda,\mxi)=\frac{\delta \abs{\mxi'}^2}
{\abs{\xi_0+\mff(t,\mx,\mlambda) \cdot \mxi'}^2
+\delta \abs{\mxi'}^2},
$$
which means that we must prove
\begin{equation}\label{v5}
	\lim\limits_{M\to \infty} \lim\limits_{\delta\to 0}
	\actionb{\chi_M(t,\mx,\mlambda)\ph(t,\mx)\rho(\mlambda) \psi(\mxi)
	H_\delta(t,\mx,\mlambda,\mxi)} {B}=0,
\end{equation} in order to conclude that
$\actionb{\varphi\otimes\rho\otimes \psi}{B}=0$ for 
all functions $\varphi \in C_c(\R^{d+1})$, 
$\rho \in C_c(\R^m)$, and $\psi\in C^{d+1}(\S^{d})$. From the latter, 
it follows immediately that $B=0$ in 
$L^r_{w\star}\bigl(\R^{d+1+m};(C^{d+1}(\S^d))'\bigr)$. 

Clearly, $\ph \rho \psi H_\delta\in L^{r'}(\R^{d+1+m};C^{d+1}(\S^d))$. 
In the above, the function $H_\delta$ is defined for $\abs{\mxi}=1$. 
However, this can be readily extended to the 
entirety of $\R^{d+1}\backslash \{0\}$ in terms of $\mxi$ as 
a homogeneous symbol of order zero:
\begin{equation}\label{eq:H-delta-on-Sd}
	\begin{split}
		H_{\delta}(t,\mx,\mlambda,\mxi)
		&=H_{\delta}
		\left(t,\mx,\mlambda,\frac{\mxi}{\abs{\mxi}}\right)
		= \frac{\delta}{\abs{\frac{\xi_0}{\abs{\mxi}}
		+\frac{\mff(t,\mx,\mlambda)\cdot \mxi'}{\abs{\mxi}}}^2
		+\delta \abs{\frac{\mxi'}{\abs{\mxi}}}^2}
		\abs{\frac{\mxi'}{\abs{\mxi}}}^2
		\\ & = \frac{\delta}
		{\abs{\bigl(1,\mff(t,\mx,\mlambda)\bigr)
		\cdot \frac{\mxi}{\abs{\mxi}}}^2
		+\delta \abs{\bigl(0,\mathbf{1}'\bigr)
		\cdot \frac{\mxi}{\abs{\mxi}}}^2}
		\abs{\bigl(0,\mathbf{1}'\bigr)
		\cdot\frac{\mxi}{\abs{\mxi}}}^2,
	\end{split}
\end{equation}
where $\mathbf{1}':=(1,\ldots,1)\in \R^d$. 
By treating the drift-vector $\mff=\mff(t,\mx,\mlambda)$ 
as an element of the Bochner space $L^q(\R^{d+1};L^q(\R^m;\R^d))$, we 
can approximate $\mff$ with simple 
functions $\mff^{(N)}$ of the following form:
\begin{equation}\label{eq:f-approx}
	\mff^{(N)} =\sum\limits_{k=1}^{N} \chi_k(t,\mx)
	\mff_k(\mlambda), \quad \chi_k(t,\mx)=\En_{A_k}, 
	\,\, \abs{A_k}<\infty, \,\,
	\mff_k\in L^q(\R^m;\R^d),
\end{equation} 
so that $\norm{\mff-\mff^{(N)}}_{L^q(\R^{d+1};L^q(\R^m;\R^d))}\to 0$ 
as $N\to \infty$, see, e.g., \cite[Lemma 1.2.19]{Hytonen:2016aa}.  
Given the compactly supported functions $\ph(t,\mx)$ and $\rho(\mlambda)$ 
in \eqref{v4}, it is enough to ask for this convergence on 
a compact set $K\times L \Subset\R^{d+1}\times \R^m$ 
for which $\supp(\ph)\subset K$ and $\supp(\rho)\subset L$:
\begin{equation}\label{eq:f-approx-conv}
	\lim_{N\to \infty} 
	\norm{\mff-\mff^{(N)}}_{L^q(K;L^q(L;\R^d))}=0.
\end{equation}
Without loss of generality, we may assume that the sets 
$A_1,\ldots,A_N$ are disjoint and constitute a 
partition of $K$ such that
\begin{equation}\label{eq:part-unity}
	\sum_{k=1}^N \chi_k(t,\mx)=1, \quad 
	\text{for all $(t,\mx)\in K$}.
\end{equation}
In addition, we may always assume that for any $(t,\mx)\in K$ we have
\begin{equation}\label{bemerkung}
	\text{$\abs{\mff_k(\mlambda)} 
	\leq \abs{\mff(t,\mx,\mlambda)}$
	 a.e. on  $\seq{\chi_k(t,\mx)=1}$}.
\end{equation}

Given $\seq{f^{(N)}}_{N\in\N}$, let us approximate 
$H_\delta$  by a sequence $\seq{H_\delta^{(N)}}_{N\in \N}$ 
of functions of the form
\begin{align*}
	&H_\delta^{(N)}(t,\mx,\mlambda,\mxi)
	=\sum_{k=1}^{N} \psi_{\delta,k}(\mxi;\mlambda) \chi_k(t,\mx), 
	\\ & \quad \text{where} 
	\quad \psi_{\delta,k}(\mxi;\mlambda)
	=\frac{\delta \abs{\mxi'}^2}
	{\abs{\xi_0+\mff_{k}(\mlambda)\cdot \mxi'}^2
	+\delta \abs{\mxi'}^2}.
\end{align*} 
Since $\mxi \mapsto \psi_{\delta,k}(\mxi;\mlambda)$  
can be expressed in terms of $\mxi/\abs{\mxi}$ 
as in \eqref{eq:H-delta-on-Sd}, we may view $\psi_{\delta,k}$ 
(with $\mlambda$ fixed) as an $L^r$, $r\in (1,\infty)$, 
Fourier multiplier with a bound uniform in 
$\delta$, $k$ and $\mlambda$ (see Lemma \ref{Lp-m}).
 
Before continuing, we must verify that
\begin{equation}\label{eq:H-delta-approx}
	\lim_{N\to \infty}\norm{\chi_M \left(H_\delta
	-H_\delta^{(N)}\right)}_{L^q(K\times L;C^{d+1}(\S^d))}=0.
\end{equation}
From the definitions of $H_\delta$ and $H_\delta^{(N)}$, see 
\eqref{eq:H-delta-on-Sd}, and \eqref{eq:part-unity},
\begin{align*}
	\chi_M\pa_{\mxi}^{\malpha}
 	\! \left(H_\delta^{(N)}-H_\delta\right)
	=\chi_M(t,\mx,\mlambda)  
	\sum_{k=1}^N \chi_k(t,\mx)
	\pa_{\mxi}^{\malpha}\bigl(I_k-I\bigr)(t,\mx,\mlambda,\mxi),
\end{align*}
for any $\abs{\malpha}\leq d+1$, where
\begin{align*}
	& I_k  = 
	\frac{\delta}{\abs{\frac{\xi_0}{\abs{\mxi}}+\mff_{k}(\mlambda)\cdot 
	\frac{\mxi'}{\abs{\mxi}}}^2
	+\delta \abs{\frac{\mxi'}{\abs{\mxi}}}^2}
	 \abs{\frac{\mxi'}{\abs{\mxi}}}^2,
	\\ & 
	I = 
	\frac{\delta}{\abs{\frac{\xi_0}{\abs{\mxi}}
	+\mff(t,\mx,\mlambda)\cdot 
	\frac{\mxi'}{\abs{\mxi}}}^2
	+\delta \abs{\frac{\mxi'}{\abs{\mxi}}}^2}
	 \abs{\frac{\mxi'}{\abs{\mxi}}}^2.
\end{align*}
Clearly,
$$
I_k-I=\bigl(\mff_{k}(\mlambda)
-\mff(t,\mx,\mlambda)\bigr)\cdot 
B_k(t,\mx,\mlambda,\mxi),
$$
where
\begin{align*}
	B_k& = \frac{\frac{\mxi'}{\abs{\mxi}}
	\left(2\frac{\xi_0}{\abs{\mxi}}
	+\mff_{k}(\mlambda)\cdot \frac{\mxi'}{\abs{\mxi}}
	+\mff(t,\mx,\mlambda)\cdot \frac{\mxi'}{\abs{\mxi}}\right)
	\delta \abs{\frac{\mxi'}{\abs{\mxi}}}^2}
	{\left(\abs{\frac{\xi_0}{\abs{\mxi}}
	+\mff_{k}(\mlambda)\cdot \frac{\mxi'}{\abs{\mxi}}}^2
	+\delta \abs{\frac{\mxi'}{\abs{\mxi}}}^2\right)
	\left(\abs{\frac{\xi_0}{\abs{\mxi}}
	+\mff(t,\mx,\mlambda)\cdot \frac{\mxi'}{\abs{\mxi}}}^2
	+\delta \abs{\frac{\mxi'}{\abs{\mxi}}}^2\right)},
\end{align*}
so that
\begin{equation}\label{eq:H-N-delta-minus-H-delta}
	\begin{split}
		& \chi_M(t,\mx,\mlambda)
 		\pa_{\mxi}^{\malpha}
 		\! \left(H_\delta^{(N)}-H_\delta\right)
		\\ & \quad =\chi_M(t,\mx,\mlambda)  
		\sum_{k=1}^N \chi_k(t,\mx)
		\bigl(\mff_{k}(\mlambda)-\mff(t,\mx,\mlambda)\bigr)
		\cdot \pa_{\mxi}^{\malpha}B_k(t,\mx,\mlambda,\mxi),
	\end{split}
\end{equation}
for any $\abs{\malpha}\leq d+1$. We can write 
$B_k= B_k^{(1)}B_k^{(2)}B_k^{(3)}$,
where $B_k^{(3)}=\delta \abs{\frac{\mxi'}{\abs{\mxi}}}^2$ and
\begin{align*}
	B_k^{(1)} 
	& =\frac{\frac{\mxi'}{\abs{\mxi}}}
	{\abs{\bigl(1,\mff_k\bigr)
	\cdot \frac{\mxi}{\abs{\mxi}}}^2
	+\delta \abs{\bigl(0,\mathbf{1}'\bigr)
	\cdot \frac{\mxi}{\abs{\mxi}}}^2},
	\\  
	B_k^{(2)} 
	& = \frac{2\frac{\mxi_0}{\abs{\mxi}}+\bigl(\mff_{k}+\mff\bigr)
	\cdot \frac{\mxi'}{\abs{\mxi}}}
	{\abs{\bigl(1,\mff\bigr)
	\cdot \frac{\mxi}{\abs{\mxi}}}^2
	+\delta \abs{\bigl(0,\mathbf{1}'\bigr)
	\cdot \frac{\mxi}{\abs{\mxi}}}^2}.
\end{align*}
Below we will verify that, for $j=1,2,3$,
\begin{equation}\label{eq:B-j-k}
	\abs{\chi_M(t,\mx) B_k^{(j)}(t,\mx,\mlambda,\mxi)}
	\lesssim_{M,\delta} 1, 
	\quad  (t,\mx,\mlambda)\in K\times L, 
	\, \abs{\mxi}=1.
\end{equation}
In fact, for any $\abs{\malpha}\leq d+1$, we have 
a similar bound for the $\malpha$-derivative in $\mxi$ : 
\begin{equation}\label{eq:B-j-k-deriv}
	\abs{\chi_M(t,\mx)\pa_{\mxi}^{\malpha}
	B_k^{(j)}(t,\mx,\mlambda,\mxi)}
	\lesssim_{M,\delta,\abs{\malpha}} 1,
	\quad 
	 (t,\mx,\mlambda)
	\in K\times L, \, \abs{\mxi}=1.
\end{equation}
Assuming that \eqref{eq:B-j-k-deriv} is true, we deduce 
from \eqref{eq:H-N-delta-minus-H-delta} and 
\eqref{eq:f-approx-conv}, see also \eqref{eq:part-unity}, that 
\eqref{eq:H-delta-approx} holds.

To verify \eqref{eq:B-j-k}, it is enough to establish the bound
\begin{equation}\label{eq:chi-M-psi-delta-k-bounded}
	\frac{\chi_M(t,\mx,\mlambda)}
	{\abs{\frac{\xi_0}{\abs{\mxi}}+\mff_{k}(\mlambda)
	\cdot \frac{\mxi'}{\abs{\mxi}}}^2
	+\delta \abs{\frac{\mxi'}{\abs{\mxi}}}^2}
	\leq C_{M,\delta},
	\quad  (t,\mx,\mlambda,\mxi)
	\in K\times L \times (\R^{d+1}\backslash \{0\}),
\end{equation}
and the same bound when $\mff_k(\mlambda)$ is replaced 
by $\mff(t,\mx,\mlambda)$. 
Here, $C_{M,\delta}>0$ is a constant depending on the 
previously fixed (large) number $M$ and the small number $\delta>0$.
To establish \eqref{eq:chi-M-psi-delta-k-bounded}, we will 
examine two cases:
$$
\textrm{(i)} \quad 
\abs{\frac{\mxi'}{\abs{\mxi}}}
>\frac{1}{M^2}; 
\qquad 
\textrm{(ii)} \quad 
\abs{\frac{\mxi'}{\abs{\mxi}}}
\leq \frac{1}{M^2}.
$$ 
By \eqref{bemerkung}, in the first case (i), we have
\begin{equation*}
	\frac{\chi_M(t,\mx,\mlambda)}
	{\abs{\frac{\xi_0}{\abs{\mxi}}+\mff_{k}(\mlambda)
	\cdot \frac{\mxi'}{\abs{\mxi}}}^2
	+\delta \abs{\frac{\mxi'}{\abs{\mxi}}}^2}
	<\frac{M^4}{\delta}.
\end{equation*}
For the second case we apply the elementary inequality 
$b^2 \geq \frac{1}{2} (a+b)^2 - a^2$, with 
$a=-\mff_{k}(\mlambda)\cdot \mxi'/\abs{\mxi}$ and $b=\xi_0/\abs{\mxi}
+\mff_{k}(\mlambda)\cdot \mxi'/\abs{\mxi}$ 
such that $a+b=\xi_0/\abs{\mxi}$, which supplies the estimate
\begin{equation*}
	\frac{\chi_M(t,\mx,\mlambda)}
	{\abs{\frac{\xi_0}{\abs{\mxi}}+\mff_{k}(\mlambda)
	\cdot \frac{\mxi'}{\abs{\mxi}}}^2
	+\delta \abs{\frac{\mxi'}{\abs{\mxi}}}^2}
	\leq \frac{\chi_M(t,\mx,\mlambda)}
	{\frac{1}{2}\abs{\frac{\xi_0}{\abs{\mxi}}}^2
	-\abs{\mff_{k}(\mlambda)\cdot \frac{\mxi'}{\abs{\mxi}}}^2}.
\end{equation*}
Exploiting that 
$\frac{\xi_0^2}{\abs{\mxi}^2}
=1-\frac{\abs{\mxi'}^2}{\abs{\mxi}^2}
\overset{\textrm{(ii)}}\geq 1-\frac{1}{M^4}$ 
and, since \eqref{bemerkung} implies 
$\abs{f_k(\mlambda)}\leq M$,
$$
-\abs{\mff_{k}(\mlambda)\cdot \frac{\mxi'}{\abs{\mxi}}}^2
\geq -M^2\abs{ \frac{\mxi'}{\abs{\mxi}}}^2
\overset{\textrm{(ii)}}\geq -\frac{1}{M^2},
$$
we arrive at (recall $M>\sqrt{3}$)
\begin{equation*}
	\frac{\chi_M(t,\mx,\mlambda)}
	{\abs{\frac{\xi_0}{\abs{\mxi}}+\mff_{k}(\mlambda)
	\cdot \frac{\mxi'}{\abs{\mxi}}}^2
	+\delta \abs{\frac{\mxi'}{\abs{\mxi}}}^2}
	\leq \frac{1}{\frac12-\frac{1}{2M^4}-\frac{1}{M^2}}
	\leq \frac{2M^2}{M^2-3}.
\end{equation*} 
This ends the proof of \eqref{eq:chi-M-psi-delta-k-bounded}.  
We can use the exact same proof, but replace 
every instance of $\mff_k$ with $\mff$, to show that 
\eqref{eq:chi-M-psi-delta-k-bounded} is also true when 
we substitute $\mff_k$ with $\mff$.

To verify \eqref{eq:B-j-k-deriv}, note that the 
denominators in $B_k^{(j)}$ are limited in magnitude. 
As a result, if we take the $\mxi$-derivative of 
order $\abs{\malpha}\leq d+1$, the resulting expression takes the form 
$F_{\malpha}\left(t,\mx,\mlambda,\mxi/\abs{\mxi}\right)
/\abs{\mxi}^{\abs{\malpha}}$, 
where $F_{\malpha}$ is some bounded function 
(arguing as before), smooth in $\mxi$. 
We are interested in finding the $L^\infty$ norm in $\mxi$ 
on the unit sphere, which means 
\begin{align*}
	\norm{\pa_{\mxi}^{\malpha}B_k^{(j)}}_{L^\infty_{t,\mx,\mlambda,\mxi}}
	&=\norm{\frac{1}{\abs{\mxi}^{\abs{\malpha}}}F_{\malpha}
	\left(t,\mx,\mlambda,\frac{\mxi}{\abs{\mxi}}
	\right)}_{L^\infty_{t,\mx,\mlambda,\mxi}}
	\\ & 
	=\norm{F_{\malpha}
	\left(t,\mx,\mlambda,\mxi\right)}_{L^\infty_{t,\mx,\mlambda,\mxi}}
	\le C_{M,\delta,\alpha},
\end{align*}
for some constant $C_{M,\delta,\abs{\alpha}}$. 
This leads us to the assertion \eqref{eq:B-j-k-deriv}.

We will now investigate the action of the $H$-distribution $B$ 
on $\chi_M\ph\rho \psi  H_\delta^{(N)}$. For any $\ell>0$, let 
us consider the (truncation) functions
\begin{equation*}
	\begin{split}
		& T_\ell(z)=
		\begin{cases}
			z, & \text{if $\abs{z}<\ell$} \\
			0,& \text{otherwise}
		\end{cases},
		\qquad 
		T^\ell(z)=z-T_\ell(z)=
		\begin{cases}
			0, & \text{if $\abs{z}<\ell$} \\
			z,& \text{otherwise}
		\end{cases}.
	\end{split}
\end{equation*} 

By referring back to the defining equation for the $H$-distribution 
$B$, see \eqref{v2} or \eqref{rev1}, we can 
examine $\action{\chi_M \ph \rho \psi H_\delta^{(N)}}{B}$ 
using the following approach: 
\begin{align}
	\notag 
	& \action{\chi_M \ph \rho \psi H_\delta^{(N)}}{B}
	=\action{\,\chi_M(t,\mx,\mlambda) \sum_{k=1}^N 
	\varphi(t,\mx)\rho(\mlambda) \psi(\mxi)
	\psi_{\delta,k}(\mxi;\mlambda)\chi_k(t,\mx)}{B}
	\\ \notag
	& \quad
	=\lim_{n\to \infty}\int_{\R^{d+1+m}}
	\chi_M(t,\mx,\mlambda) \sum_{k=1}^N 
	\varphi(t,\mx)\rho(\mlambda)\tilde{\varphi}(t,\mx)\tilde{\rho}(\mlambda) 
	u_n(t,\mx,\mlambda)\chi_k(t,\mx)
	\\ \label{eq:HdeltaN-B}
	& \quad \qquad\qquad\qquad \qquad  \qquad \times 
	\overline{\cA_{\bar \psi\bigl(\frac{\mxi}{\abs{\mxi}}\bigr)
	\psi_{\delta,k}\bigl(\frac{\mxi}{\abs{\mxi}};\mlambda\bigr)}
	\left(v_n\right)(t,\mx)} 
	 \, d\mlambda\, d\mx \, dt
	\\ \notag 
	& \quad
	= \lim_{n\to \infty}
	\int_{\R^{d+1+m}}\chi_M(t,\mx,\mlambda)\sum_{k=1}^N 
	\varphi(t,\mx)\rho(\mlambda) \tilde{\varphi}(t,\mx)
	\tilde{\rho}(\mlambda) T_\ell(u_n(t,\mx,\mlambda))\chi_k(t,\mx)  
	\\ \notag 
	& \quad\qquad \qquad \qquad \qquad \qquad \times 
	\overline{\cA_{\bar \psi\bigl(\frac{\mxi}{\abs{\mxi}}\bigr)
	\psi_{\delta,k}\bigl(\frac{\mxi}{\abs{\mxi}};\mlambda\bigr)}
	\left(v_n\right)(t,\mx)} \,d\mlambda \,d\mx \,dt   
	\\ \notag 
	& \quad\quad\quad
	+\lim_{n\to \infty}\int_{\R^{d+1+m}} 
	\chi_M(t,\mx,\mlambda)\sum_{k=1}^N 
	\varphi(t,\mx)\rho(\mlambda) \tilde{\varphi}(t,\mx)\tilde{\rho}(\mlambda) 
	T^\ell(u_n(t,\mx,\mlambda))
	\\ \notag 
	& \quad\qquad \qquad \qquad \qquad \qquad \times 
	\chi_k(t,\mx)
	\overline{\cA_{\bar \psi\bigl(\frac{\mxi}{\abs{\mxi}}\bigr)
	\psi_{\delta,k}\bigl(\frac{\mxi}{\abs{\mxi}};\mlambda\bigr)}
	\left(\chi_k v_n\right)(t,\mx)}
	\, d\mlambda \, d\mx \, dt
	\\ \notag 
	& \quad  
	=: \lim_{n\to \infty} I_\ell(n,N,\delta,M)
	+\lim_{n\to \infty} I^\ell(n,N,\delta,M), 
	\quad \text{for any $l>0$}.
\end{align}
Here, we have utilised the fact that $\chi_k=\chi_k^2$ and the 
commutation lemma (Lemma \ref{lem:commutation}) to move one 
instance of $\chi_k$ inside $\cA$, arriving at 
the third equality. We note that, according to Lemma \ref{Lp-m}, the 
function $\psi_{\delta,k}$ qualifies as an $L^r$ Fourier 
multiplier for all $r \in (1, \infty)$, with a bound that 
holds uniformly with respect to $\mlambda$.

Observe that the $n$-limit of $I_\ell(n,N,\delta,M)$ 
(for fixed $\ell$) can be expressed using 
the ``standard" vector-valued $H$-measure $\mu_l$ 
(see Theorem \ref{thm:gen-H-measure}), 
which is generated by the sequences 
$\seq{\tilde{\rho} \tilde{\varphi}T_\ell(u_n)}$ 
and $\seq{v_n}$ of compactly supported $L^\infty$ functions: 
\begin{align*}
	\lim_{n\to \infty} I_\ell(n,N,\delta,M) 
	=\action{\chi_M \ph \rho \psi H_\delta^{(N)}}{\mu_\ell}.
\end{align*}
Moreover, the $H$-measure $\mu_\ell$ is a bounded bilinear 
functional that belongs to
$L^2_{w\star}\bigl(\R^{d+1+m};\cM(\S^d)\bigr)$.
Given the approximation result \eqref{eq:H-delta-approx}, 
it thus follows that
\begin{equation}\label{eq:HdeltaN-muell-approx}
	\begin{split}
		&\abs{\actionb{\chi_M \ph\rho\psi H_\delta}{\mu_\ell}
		-\action{\chi_M \ph\psi H_\delta^{(N)}}{\mu_\ell}}
		=\abs{\action{\chi_M \ph \rho  \psi 
		\left(H_\delta-H_\delta^{(N)}\right)}{\mu_\ell}}
		\\ & \qquad 
		\leq C_\ell \norm{\chi_M \ph \rho \psi 
		\bigl(H_\delta-H_\delta^{(N)}\bigr)}_{L^2(\R^{d+1+m};C^d(\S^d))}
		\\ & \qquad
		\leq C_{\ell,\ph,\rho,\psi,M} \norm{H_\delta
		-H_\delta^{(N)}}_{L^q(K\times L;C^{d}(\S^d))}\toN 0,
	\end{split}
\end{equation}
recalling that $2\leq q$, so that
\begin{align*}
	\lim_{N\to\infty} \lim_{n\to \infty} I_\ell(n,N,\delta,M) 
	=\actionb{\chi_M \ph \rho \psi H_\delta}{\mu_\ell}.
\end{align*}
By Lemma \ref{loc-p},
\begin{align*}
	\lim_{\delta\to 0}\actionb{\chi_M \ph \rho \psi H_\delta}{\mu_\ell}
	=\lim_{\delta\to 0}
	\action{\frac{\delta \chi_M(t,\mx,\mlambda) \varphi(t,\mx) 
	\rho(\mlambda) \psi(\mxi)\abs{\mxi'}^2}
	{\abs{\xi_0+\mff(t,\mx,\mlambda)\cdot \mxi'}^2
	+\delta \abs{\mxi'}^2}}{\mu_\ell}=0,
\end{align*} 
so that, for each fixed $\ell$,
\begin{equation}\label{eq:I-sub-ell-conv}
	\lim_{M\to \infty}\lim_{\delta\to 0}\lim_{N\to\infty} 
	\lim_{n\to \infty} I_\ell(n,N,\delta,M) =0.
\end{equation}

Let us now consider the second limit in \eqref{eq:HdeltaN-B}. 
First, recall that $\supp(\ph)\subset K\Subset \R^{d+1}$ 
and $\supp(\rho)\subset L\Subset \R^m$. 
Applying Lemma \ref{crucial} allows us to establish an upper bound 
for $I^\ell(n,N,\delta,M)$ in the following way: for $\bar{p}\in (1,p)$ 
and $\frac{1}{\bar{p}}+\frac{1}{\bar{p}'}=1$,
$$
\abs{I^\ell(n,N,\delta,M)} \leq C_{\ph, \rho} 
\norm{T^l(u_n)}_{L^{\bar{p}}(K\times L)} 
\norm{v_n}_{L^{\bar{p}'}(K\times L)},
$$
where the constant $C_{\ph, \rho} $ is independent of $n$, 
$N$, $\delta$, and $M$. By \eqref{eq:intro-ass1},
\begin{align*}
	\norm{T^l(u_n)}_{L^{\bar{p}}(K\times L)} ^{\bar{p}}
	& \leq  \int_{\abs{u_n}\ge \ell} 
	\abs{u_n}^{\bar{p}} \,d\mlambda \, d\mx \, dt 
	\leq \frac{C}{\ell^{p-\bar{p}}}
	\toell 0,
\end{align*}
and hence, uniformly in $n$, $N$, $\delta$, and $M$
\begin{equation}\label{eq:I-super-ell-conv}
	\lim_{\ell\to\infty} I^\ell(n,N,\delta,M) =0.
\end{equation}

After using \eqref{eq:H-delta-approx} and 
following the reasoning outlined in \eqref{eq:HdeltaN-muell-approx}, 
and bearing in mind that the $H$-distribution $B$ is a 
bounded bilinear functional that belongs to 
$L^r_{w\star}\bigl(\R^{d+1+m};(C^{d+1}(\S^d))'\bigr)$, 
where $q>r'=r/(r-1)$, it can be inferred that: 
\begin{equation}\label{eq:HdeltaN-B-approx}
	\abs{\actionb{\chi_M \ph\rho\psi H_\delta}{B}
	-\action{\chi_M \ph\psi H_\delta^{(N)}}{B}}
	\toN 0.
\end{equation}

Given \eqref{eq:I-sub-ell-conv}, \eqref{eq:I-super-ell-conv}, 
and \eqref{eq:HdeltaN-B-approx}, we can deduce 
from \eqref{eq:HdeltaN-B} that \eqref{v5} is valid. 
By combining \eqref{v4} with \eqref{v5}, we arrive 
at the conclusion that
$$
\actionb{\varphi\otimes \rho \otimes \psi}{B}=0, 
\quad  \varphi\otimes\rho\otimes \psi
\in C_c(\R^{d+1})\otimes C_c(\R^m)\otimes C^{d+1}(\S^{d}).
$$ 
As a result, upon selecting the symbol $\psi \equiv 1$ and $\rho$ 
such that $\rho \equiv 1$ on $\supp(\tilde{\rho})$, this 
implies the theorem's statement, particularly in reference to \eqref{res-1}.
When $\psi\equiv 1$, the multiplier operator $\cA_\psi$ 
becomes the identity operator and thus, 
recalling \eqref{eq:hn-weak-Lp}, \eqref{vn}, 
\eqref{eq:vn-weak-conv} and \eqref{rev1},
\begin{align*}
	0 &=\actionb{\varphi\otimes\rho\otimes 1}{B}
	\\ & = \lim_{n\to \infty}\int_{\R^{d+1+m}}
	\ph(t,\mx)\rho(\mlambda)\tilde{\ph}(t,\mx)
	\tilde{\rho}(\mlambda)u_n(t,\mx,\mlambda) 
	v_n(t,\mx)\, d\mlambda\, d\mx \, dt
	\\ & =\lim_{n\to \infty} 
	\int_{\R^{d+1}} \varphi(t,\mx)\tilde{\ph}(t,\mx)^2 
	\abs{\actionb{u_n}{\tilde{\rho}}(t,\mx)} \, d\mx \, dt, 
\end{align*}
which, in view of the arbitrariness of $\ph \tilde{\ph}^2$, 
implies \eqref{res-1}. 
\end{proof}

\section{Stochastic velocity averaging}\label{sec:stochastic}

\subsection{Preliminary considerations}
Recall that we consider a sequence of  kinetic SPDEs of 
the form \eqref{eq-s}. We should further clarify the assumptions regarding 
the stochastic components of \eqref{eq-s}. 
For background information on stochastic analysis and SPDEs, 
including the framework of It\^o integration with respect to 
cylindrical Wiener processes, we refer \cite{DaPrato:2014aa}. 
For properties of Bochner spaces like 
$L^p(\Omega;X)= L^p\bigl(\Omega,\cF,P;X\bigr)$, 
where $X$ is a Banach space, we refer to \cite{Hytonen:2016aa}. 

Consider a stochastic basis 
\begin{equation}\label{eq:stoch-basis}
	\cS_n=\bigl(\Omega, \cF,\seq{\cF_t^n}_{t\ge 0},P\bigr)
\end{equation}
with expectation operator denoted by $E$. 
This basis consists of a complete probability space 
$(\Omega,\cF,P)$ and a complete 
right-continuous filtration $\seq{\cF_t^n}_{t\ge 0}$. 
Let $\Bbb{H}$ be a separable Hilbert space with 
norm $\norm{\cdot}_{\Bbb{H}}$ and inner product $(\cdot,\cdot)_{\Bbb{H}}$.
Consider a second separable Hilbert space $\Bbb{K}$. 
Let $W_n(t)$ be a cylindrical Wiener process on $\cS_n$
evolving over $\Bbb{K}$, formally defined as
\begin{equation}\label{eq:Wiener-process}
	W_n(t) = \sum_{\ell=1}^{\infty} w_{n,\ell}(t)
	e_\ell, \quad t\ge 0,
\end{equation} where $\seq{w_{n,\ell}}_{\ell=1}^\infty$ are mutually 
independent standard $\R$-valued Wiener processes defined 
on $\cS_n$, and $\seq{e_{\ell}}_{\ell=1}^\infty$ is an orthonormal 
basis of $\Bbb{K}$. It is well known that the series defining $W_n$ 
does not converge in $\Bbb{K}$, but rather in any Hilbert space 
$\tilde{\Bbb{K}}\supset \Bbb{K}$ with a Hilbert-Schmidt embedding. 
We use $L_2(\Bbb{K},\Bbb{H})$ to denote the space of 
Hilbert-Schmidt operators mapping from $\Bbb{K}$ to $\Bbb{H}$.

We call $\Phi_n:\Omega \times [0, T]\to L_2(\Bbb{K},\Bbb{H})$ a
progressively measurable stochastic process if the restriction 
of $\Phi_n$ to $\Omega \times [0, t]$ is 
$\cF_t \times \cB([0,t])$-measurable for every $t \in [0, T]$. 

For a progressively measurable process 
$\Phi_n \in L^2(\Omega \times [0, T];L_2(\Bbb{K},\Bbb{H}))$, we 
define the \textit{It\^{o} stochastic integral} $I_n(\Phi_n)$ by
\begin{equation}\label{eq:stoch-int-n}
	I_n(\Phi_n)(t) = \int_0^t \Phi_n(s) \, dW_n(s), \quad t\in [0,T],
\end{equation}
where $I_n(\Phi_n)$ is a progressively 
measurable process in $L^2(\Omega \times [0, T];\Bbb{H})$.  
The It\^{o} integral  $I_n(\Phi_n)$ 
is the unique continuous $\Bbb{H}$-valued martingale such that 
for any $h \in \Bbb{H}$ and $t \in [0, T]$,
$$
\bigl( I_n(\Phi_n)(t),h\bigr)_{\Bbb{H}} 
=\sum_{\ell=1}^\infty \int_0^t \bigl(\Phi_n(s)e_\ell, h\bigr)_{\Bbb{H}}
\, dw_{n,\ell}(s),
$$
where each real-valued stochastic integral in the series 
is understood in the sense of It\^{o}. The convergence of the series 
takes place in $L^2(\Omega,\cF_t,P;C[0, t])$ for every $t\in [0,T]$. 
For a comprehensive explanation, see \cite{DaPrato:2014aa}. 

In the subsequent discussion, we will specifically assume that
\begin{equation}\label{eq:H-space}
	\Bbb{H}=L^2(\R^{d+m})
	=L^2(\R^d\times \R^m)=:L^2_{\mx,\mlambda}.
\end{equation}
To derive the strong compactness of 
velocity averages of the kinetic SPDEs \eqref{eq-s}, we first 
need to elucidate a suitable interpretation of \eqref{eq-s}. This should 
be done in a manner that necessitates minimal assumptions. 
For context, let us revisit the conventional interpretation that requires 
more stringent assumptions. 
We will then naturally transition to the interpretation 
tailored for the weaker assumptions. Therefore, following 
\cite{Langa:2003aa}, let us momentarily assume 
that the solutions $u_n$ belong to $L^2\bigl(\Omega,\cF_t,P;
L^\infty([0,t];L^2_{\mx,\mlambda})\bigr)$ for every $t\in [0,T]$ 
and that, for every spatial test function 
$v\in \cD_{\mx,\mlambda}:=\cD(\R^d\times \R^m)$, 
$t\mapsto \int_{\R^{d+m}} u_n(t) v \, d\mlambda\, d\mx$ 
is a.s.~continuous on $[0,T]$. Furthermore, assume that 
$\Phi_n\in L^2(\Omega \times [0,T];L_2(\Bbb{K},L^2_{\mx,\mlambda}))$
is a progressively measurable process, and that $u_n$ 
satisfies the SPDE \eqref{eq-s} in the distributional (weak) 
sense in $\mx, \mlambda$: for every $v\in \cD_{\mx,\mlambda}$ and 
for every $0\leq s<t\leq T$, the following equation holds a.s.: 
\begin{equation}\label{weak-sense-xl-tmp1}
	\begin{split}
		& \bigl( u_n(t),v\bigr)_{L^2_{\mx,\mlambda}}
		-\bigl( u_n(s),v\bigr)_{L^2_{\mx,\mlambda}}
		-\int_s^t \Bigl(\mff(t',\cdot_\mx,\cdot_\mlambda)u_n(t'), 
		\nabla_\mx v \Bigr)_{L^2_{\mx,\mlambda}}\, dt'
		\\ & \qquad
		= (-1)^{N_\mlambda} \int_s^t \innb{g_n(t'),
		\pa_{\mlambda}^{\malpha}v}_{W^{-1,\alpha}_{\mx,\mlambda},
		W^{1,\alpha'}_{\mx,{\mlambda}}} \,dt'
		\\ & \qquad\qquad 
		+\bigl(I_n(\Phi_n)(t),v\bigr)_{L^2_{\mx,\mlambda}}
		-\bigl(I_n(\Phi_n)(s),v\bigr)_{L^2_{\mx,\mlambda}},
	\end{split}	
\end{equation}
where $\inn{\cdot,\cdot}_{W^{-1,\alpha}_{\mx,\mlambda},
W^{1,\alpha'}_{\mx,{\mlambda}}}$---occasionally written 
simply as $\inn{\cdot,\cdot}$---denotes the 
duality pairing between $W^{-1,\alpha}(\R^{d+m})$ and 
$W^{1,\alpha'}(\R^{d+m})$, with $\alpha>1$ and 
$\frac{1}{\alpha}+\frac{1}{\alpha'}=1$. 
Here, to make sense of \eqref{weak-sense-xl-tmp1}, 
we also assume that $(\omega,t)\mapsto \inn{g_n(t),
\pa_{\mlambda}^{N_\mlambda}v}$ is integrable 
on $\Omega\times [0,T]$ (w.r.t.~$dP\times dt$). 

As the time derivative map $\pa_t$ is continuous from  
$C$ (and thus $L^\infty$) to $W^{-1,\infty}$ and $I_n(\Phi_n)$ belongs to 
$L^2\bigl(\Omega,\cF_t,P;C([0,t];L^2_{\mx,\mlambda})\bigr)$ 
for every $t\in [0,T]$,
\begin{equation}\label{eq:stoch-int-n-deriv}
	\pa_t I_n(\Phi_n)\in 
	L^2\bigl(\Omega,\cF_t,P;W^{-1,\infty}([0,t];L^2_{\mx,\mlambda})\bigr),
	\quad  t\in [0,T].
\end{equation}
At a later stage, we derive improved estimates on $\pa_t I_n(\Phi_n)$. 
Following the notation in \cite{Langa:2003aa}, $\pa_t I_n(\Phi_n)$ 
would be formally denoted as $\Phi_n \dot{W}_n$. However, we will 
maintain the use of $\pa_t I_n(\Phi_n)$. When we keep $\omega\in \Omega$ 
fixed and differentiate \eqref{weak-sense-xl-tmp1} in $t$, we 
obtain the following equation that holds 
in the distributional sense on $(0,T)$:
\begin{equation}\label{weak-sense-xl-tmp2}
	\begin{split}
		& \bigl( \pa_t u_n,v\bigr)_{L^2_{\mx,\mlambda}}
		-\Bigl(\mff(\cdot_t,\cdot_\mx,\cdot_\mlambda)
		u_n,\nabla_x v \Bigr)_{L^2_{\mx,\mlambda}}
		\\ & \qquad 
		= (-1)^{N_\mlambda} \innb{g_n(\cdot_t,\cdot_\mx,\cdot_\mlambda),
		\pa_{\mlambda}^{\malpha}v}_{W^{-1,\alpha}_{\mx,\mlambda},
		W^{1,\alpha'}_{\mx,{\mlambda}}}
		\\ & \qquad\qquad 
		+\bigl(\pa_t I_n(\Phi_n),v\bigr)_{L^2_{\mx,\mlambda}}
		\quad \text{in $\cDp(0,T)$},
	\end{split}	
\end{equation}
where $\pa_t u_n\in L^2\bigl(\Omega,\cF_t,P;
W^{-1,\infty}([0,t];L^2_{\mx,\mlambda})\bigr)$ 
for every time $t\in [0,T]$. 
With our assumptions, 
$\mff(\cdot_t,\cdot_\mx,\cdot_\mlambda) u_n$ 
and $\pa_{\mlambda}^{\malpha}g_n$ belong to 
$L^1\bigl(\Omega,\cF_t,P;
L^1([0,t];\cDp_{\mx,\mlambda})\bigr)$ and 
$L^1[0,t]\subset W^{-1,\infty}[0,t]$. Set 
$$
h_n :=\pa_t u_n+\Div_\mx 
\bigl (\, \mff(\cdot_t,\cdot_\mx,\cdot_\mlambda) 
u_n\, \bigr)-\pa_{\mlambda}^{\malpha} g_n
-\pa_t I_n(\Phi_n).
$$
Then, for all $t\in [0,T]$, 
$$
h_n\in L^1\bigl(\Omega,\cF_t,P;
W^{-1,\infty}([0,t];\cDp_{\mx,\mlambda})\bigr)
\subset L^1\bigl(\Omega,\cF_t,P;
\cDp([0,t];\cDp(\R^d\times \R^m))\bigr)
$$ 
and \eqref{weak-sense-xl-tmp2} implies that
\begin{equation*}
	\innb{h_n,v}_{\cDp_{\mx,\mlambda},\cD_{\mx,\mlambda}}=0
	\quad \text{in $\cDp(0,T)$}, 
	\quad \text{$v\in \cD_{\mx,\mlambda}$}.
\end{equation*}
In other words, a.s., 
$$
h_n=0 
\quad 
\text{in $\cDp(0,T;\cDp(\R^d\times\R^m))
=\cDp((0,T)\times \R^d\times \R^m)
=:\cDp_{t,\mx,\mlambda}$}.
$$
See \cite{Langa:2003aa} for further details.

\medskip

In summary, given the added 
assumptions in this section, the It\^{o} equations 
\eqref{weak-sense-xl-tmp1} imply that the equations below hold 
weakly on $(0,T)\times \R^{d+m}$:
\begin{equation}\label{weak-sense-txl-tmp1}
	\pa_t u_n+\Div_\mx 
	\bigl (\, \mff(t,\mx,\mlambda) u_n\, \bigr)
	=\pa_{\mlambda}^{\malpha} g_n+\pa_t I_n(\Phi_n)
	\quad \text{in $\cDp_{t,\mx,\mlambda}$, a.s.}
\end{equation}
In our approach, we will commence not from 
\eqref{weak-sense-xl-tmp1} but \eqref{weak-sense-txl-tmp1}. 
This will serve as the foundational point for formulating minimal 
assumptions and consequently obtaining a general compactness 
framework for velocity averages in the stochastic context. 
In applications, as evidenced in \cite{Karlsen:2022aa}, there is a 
tendency to consider approximate solutions 
where \eqref{weak-sense-xl-tmp1} is valid. 
The outline above shows how to shift from \eqref{weak-sense-xl-tmp1} 
to the $(t,\mx,\mlambda)$-weak formulation \eqref{weak-sense-txl-tmp1}, 
with the latter being instrumental in proving the 
velocity averaging result. Yet, before embarking on this, we need a 
stochastic version of the $H$-distribution concept.

\subsection{Stochastic $H$-distributions}

We need a stochastic adaptation of 
the concept of $H$-distributions. Although this adaptation is merely a 
subtle variation of the deterministic $H$-distributions, it is 
still unexplored territory. As a result, we will provide its statement and 
proof within this subsection rather than relegating it to the appendix.

In the theorem below, we examine a sequence $\seq{v_n}$ of functions 
defined on the product space $\Omega\times [0,T]\times \R^{d+m}$. 
In many applications, the most 
direct approach is to ascertain the weak convergence of $v_n$ 
across all variables $(\omega,t,\mx,\mlambda)$ on the product space, 
encompassing the probability variable $\omega$. Nevertheless, by 
leveraging an appropriate implementation of the Skorokhod--Jakubowski 
theorem \cite{Jakubowski:1997aa}, we 
may assume weak convergence with respect to the 
variables $(t,\mx,\mlambda)$, for $P$-a.e.~$\omega\in \Omega$, see 
the upcoming assumption \eqref{conv-ass} as well as
the assumptions of Theorem \ref{thm:stoch-velocity}. 
For a more in-depth discussion of this issue, see \cite{Karlsen:2022aa}.

\begin{theorem}\label{thm:H-distr-s}
Let $\seq{u_n}=\seq{u_n(\omega,t,\mx,\mlambda)}$ 
be a sequence of functions defined 
on $\Omega\times[0,T]\times \R^{d+m}$ 
satisfying $E\norm{u_n}_{L^p_{t,\mx,\mlambda}}
\lesssim 1$, for some $p>1$, and be uniformly compactly 
supported with respect to $(\mx,\mlambda)\in \R^d\times \R^m$ 
in a set $K\times L\Subset \R^{d+m}$.  

Let $\seq{v_n}=\seq{v_n(\omega,t,\mx)}$ 
be a sequence of functions defined 
on $\Omega\times[0,T]\times \R^{d}$ that is uniformly 
compactly supported with respect to 
$\mx$ in a set $K\Subset \R^d$, satisfying 
$\norm{v_n}_{L^\infty_{\omega,t,\mx}}\lesssim 1$ and
\begin{equation}\label{conv-ass}
	 v_n \weakn 0  
	 \quad \text{weakly-$\star$ in $L^\infty([0,T]\times\R^d)$, a.s.}
\end{equation}

For any $r' > 1$ that satisfies $\frac{1}{p} + \frac{1}{r'} < 1$, 
there exist subsequences of $\seq{u_n}$ and $\seq{v_n}$ (without 
relabelling) accompanied by a deterministic functional
$$
B=B(t,\mx,\mlambda,\mxi)
\in \bigl( L^{r'}([0,T]\times \R^{d+m})\times  
C^{d+1}(\S^d) \bigr)',
$$ 
where $\mxi=(\xi_0,\dots,\xi_d)
:=(\xi_0,\mxi')\in \bS^d$, such that 
\begin{align}\label{rev11}
	\begin{split}
		& \inn{\varphi \otimes \psi,B}
		\\ & \quad 
		=\lim\limits_{n\to\infty}
		E \Biggl[ \, \int_0^T\int_{\R^{d+m}} 
		\varphi(t,\mx, \mlambda)
		u_n(\cdot_\omega,t,\mx,\mlambda)
		\overline{ \cA_{\overline{\psi}}(v_n)(\cdot_\omega,t,\mx)}
		\, d\mlambda \, d\mx \, dt\Biggr],
	\end{split}	
\end{align}
for any $\varphi \in L^{r'}([0,T]\times \R^d\times\R^m)$ 
and $\psi\in C^{d+1}(\bS^d)$.

Occasionally, we refer to $B$ as the stochastic $H$-distribution 
associated with the subsequences $\seq{u_n}$ and $\seq{v_n}$.
\end{theorem}

\begin{remark}\label{rem-all}
From the theorem's proof, we may also assume that the test 
function (multiplier) $\psi$ in \eqref{rev11} can depend 
on $\mlambda \in \mathbb{R}^m$.
\end{remark}

\begin{proof}
To streamline our notation, we omit the explicit mention 
of integration domains. They can be inferred from the context 
and the corresponding differential. Additionally, we note 
that functions defined for $t \in [0,T]$ are extended to $\R$ by 
assigning them a value of zero for $t\notin [0,T]$. 

For $\psi\in C^{d+1}(\bS^d)$ and 
$\varphi \in L^{r'}([0,T]\times \R^{d}\times\R^m)$, 
consider the functionals 
\begin{equation}\label{step1}
	\begin{split}
		& \innb{\varphi \otimes \psi,B_n} 
		\\& \quad 
		=\int\varphi(t,\mx,\mlambda)
		u_n(\omega,t,\mx,\mlambda)
		\overline{\cA_{\bar \psi(\mxi/\abs{\mxi})}(v_n)(\omega,t,\mx)}
		\, d\mlambda \, d\mx \, dt \, dP(\omega),
	\end{split}
\end{equation} 
The functionals $B_n$, $n\in \N$, are clearly linear 
on $L^{r'}(\R^{d+1+m}) \otimes C^{d+1}(\bS^d)$. 
They are also uniformly bounded on the same space. 
Indeed, by H\"older's inequality and 
the Marcinkiewicz theorem (see 
Theorem \ref{thm:multipliers}),
\begin{equation*}
	\begin{split}
		& \abs{\innb{\varphi \otimes \psi,B_n}}
		\\ & \quad 
		\leq  \int \abs{\varphi(t,\mx,\mlambda)
		u_n(\omega,t,\mx,\mlambda)}
		\abs{\cA_{\bar \psi(\mxi/\abs{\mxi})}(v_n)(\omega,t,\mx)} 
		\, d\mlambda  \, d\mx \, dt \, dP(\omega)
		\\ & \quad 
		\lesssim_{L,d,q}
		\norm{\varphi}_{L^{r'}([0,T]\times \R^{d+m})}
		\norm{\psi}_{C^{d+1}(\bS^d)} 
		\underbrace{E\left[\norm{u_n}_{L^p([0,T]\times K\times L)} 
		\norm{v_n}_{L^q([0,T]\times K)}\right]}_{=:J_n},
	\end{split}
\end{equation*} whereas $q$ is a number 
satisfying $\frac{1}{r'}+\frac{1}{p}+\frac{1}{q}=1$. 
By our assumptions on $u_n$ and $v_n$, we have
$J_n\lesssim_{T,K} 1$. Thus, 
$$
\seq{B_n}\subset  
\bigl( L^{r'}([0,T]\times \R^{d+m})\times
C^{d+1}( \bS^d)\bigr)'
$$ 
is a sequence of linear functionals that is bounded
 independently of $n$. According to the 
 Banach-Alaoglu theorem, we 
conclude that there exists a subsequence 
of $\seq{B_n}$ (not relabelled) and a 
limit functional $B$ such 
that $B_n \overset{\star}{\dscon} B$ as $n\to \infty$, which---via 
\eqref{step1}---implies that $B$ 
satisfies \eqref{rev11}.
\end{proof} 

To make Theorem \ref{thm:H-distr-s} useful in applications, 
we need to show that the limit functional 
$B$ can be extended to a functional that belongs to 
the dual of the Bochner space 
$L^{r'}([0,T]\times \R^{d+m};C^{d+1}(\mathbb{S}^d))$.

\begin{corollary}\label{extension} 
The deterministic linear functional $B$---defined 
by \eqref{rev11}---can be continuously extended to the 
space $\bigl(L^{r'}([0,T]\times \R^{d+m};
C^{d+1}(\bS^d))\bigr)'$, i.e., 
$$
B\in L^{r}_{w\star}\bigr([0,T]\times \R^{d+m};
(C^{d+1}(\bS^d))'\bigr).
$$
\end{corollary}

\begin{proof}
Clearly, any function $\varphi 
\in L^{r'}([0,T]\times \R^{d+m};C^{d+1}( \bS^d))$ 
can be approximated by functions of the form 
$$
\varphi^N(t,\mx,\xi)=\sum\limits_{k=1}^N \psi^N_{k}(\mxi)
\chi^N_k(t,\mx,\mlambda), \quad N=1,2,\ldots,
$$ 
where $\psi^N_{k} \in C^{d+1}(\bS^d)$, and 
$\chi^N_k(t,\mx,\mlambda)$, $k=1,\ldots,N$, 
are characteristic functions of mutually disjoint sets. 
If we establish that the functional $B$ is $N$-uniformly 
bounded in the dual space 
$\bigl( L^{r'}([0,T]\times \R^{d+m};C^{d+1}(\bS^d))\bigr)'$ 
when acting on functions of the form $\varphi^N$, then the 
conclusion of the corollary is immediate.

To show this, fix $q\in (1,\infty)$ such that 
$\frac{1}{r'}+\frac{1}{p}+\frac{1}{q}=1$.
Noting that $\chi_k^N=(\chi_k^N)^2$ and 
using the commutation lemma (Lemma \ref{lem:commutation}) 
to ``move one $\chi_k^N$ inside $\cA_{\psi_k^N}(\cdot)$", we deduce
\begin{align*}
	& \abs{\innb{\varphi^N,B}} 
	= \Biggl | \lim_{n\to\infty}
	E \Biggl [ \, \int \sum_{k=1}^N \chi^N_k(t,\mx,\mlambda)
	u_n(\cdot_\omega,t,\mx,\mlambda)
	\\ &  \quad\qquad \qquad\qquad\qquad \qquad \qquad \times
	\overline{\cA_{\bar \psi_k^N(\frac{\mxi}{\abs{\mxi}})}
	\bigl( \chi^N_k v_n \bigr)(\cdot_\omega,t,\mx)}
	\, d\mlambda \, d\mx\, dt \Biggr ] \Biggr|
	\\ & \qquad \leq \limsup_{n\to\infty} 
	\int \left(\, \sum_{k=1}^N 
	\norm{\psi_k^N}^{r'}_{C^{d+1}(\bS^d)} 
	\chi^N_k(t,\mx,\mlambda)\right)^{\frac{1}{r'}}
	\\ &  \quad \qquad \qquad \times
	\left(\, \sum_{k=1}^N \chi^N_k(t,\mx,\mlambda) 
	\abs{u_n(\omega,\mx,\mlambda)}^p \right)^{\frac1p} 
	\\ & \quad \qquad\qquad
	\times \left( \, \sum_{k=1}^N
	\frac{1}{\norm{\psi_k^N}^{q}_{C^{d+1}(\bS^d)}} 
	\abs{\cA_{\bar \psi_k^N(\frac{\mxi}{\abs{\mxi}})}
	\bigl( \chi^N_k v_n \bigr)(\omega,t,\mx)}^q\right)^{\frac1q} 
	\!\! \, d\mlambda \, d\mx\, dt\, dP,
\end{align*}
where the final step employs the discrete H\"older inequality. 
To proceed, apply the integral H\"older 
inequality and the Marcinkiewicz multiplier 
theorem (Theorem \ref{thm:multipliers}) to obtain
\begin{align*}
	\abs{\innb{\varphi^N,B}} 
	& \lesssim_{d,q,L} \norm{\sum\limits_{j=1}^N
	\norm{\psi_k^N}_{C^{d+1}(\bS^d)}
	\chi^N_k}_{L^{r'}([0,T]\times K\times L)} 
	\\ & 
	\quad
	\times \limsup_{n\to\infty} 
	E\left[\norm{\sum\limits_{k=1}^N 
	\chi^N_k u_n}_{L^p([0,T]\times K\times L)} 
	\norm{\sum\limits_{k=1}^N 
	\chi^N_k v_n}_{L^q([0,T]\times K \times L)}\right]
	\\ & \lesssim_{d,q,T,K,L} 
	\norm{\sum_{k=1}^N  \psi_k^N
	\chi^N_k}_{L^{r'}([0,T]\times K\times L;C^{d+1}(\bS^d))}
	\\ & 
	=\norm{\varphi^N}_{L^{r'}([0,T]\times 
	K\times L;C^{d+1}(\bS^d))},
\end{align*} where have used the uniform $L^p$ and $L^\infty$ 
bounds on $u_n$ and $v_n$, respectively. 
The obtained bound $\abs{\innb{\varphi^N,B}}\lesssim_{d,q,L} 
\norm{\varphi^N}_{L^{r'}([0,T]\times\R^{d+m};C^{d+1}(\bS^d))}$ 
is $N$-independent.
\end{proof} 

We observe that the stochastic analogue of the generalized 
$H$-measures from Theorem \ref{thm:gen-H-measure} 
remains valid. It suffices to incorporate 
expectation operator on the right-hand side 
of the limit  \eqref{mu-repr-h}, and update the assumptions 
of the theorem to include the probability variable. 
More precisely, the following theorem holds (its 
proof is omitted for brevity):

\begin{theorem}\label{thm:gen-H-measure-s}
Assume that the conditions of Theorem \ref{thm:H-distr-s} 
hold with $p>2$ (instead of $p>1$). Let $r'>1$ 
be such that $\frac{1}{p}+\frac{1}{r'}<1$. Then 
there exists a subsequence (not relabelled) 
and a continuous functional 
$$\mu\in 
L^2_{w\star}\bigl(K_\mlambda;
\cM([0,T]\times K_\mx\times \S^{d})\bigr) 
\bigcap L^{r}_{w\star}
\bigl([0,T]\times K_\mx\times K_\mlambda;
(C^{d+1}(\S^{d})'\bigr)
$$ 
such that for every
$\varphi\in L^{r'}([0,T]\times K_\mx\times K_\mlambda)$ 
and $\psi\in{\rm C}^{d+1}(\S^{d})$,
\begin{equation*}
	\mu(\varphi \psi) = 
	\lim_{n\to\infty}E\left[\,\,
	\int\limits_{[0,T]\times K_\mx\times K_\mlambda}
	\varphi(t,\mx,\mlambda) u_n(t,\mx,\mlambda)
	\overline{\cA_{\bar{\psi}(\mxi/|\mxi|)}(v_n)(t,\mx)}
	\,d\mlambda\, d\mx\, dt \right].
\end{equation*}  
Furthermore, $\mu$ has the representation 
\begin{equation*}
	\mu(t,\mx,\mlambda,\mxi)=
	g(t,\mx,\mlambda,\mxi) \, d\mlambda 
	\, d\nu(t,\mx,\mxi), 
	\quad 
	g\in L_\nu^1([0,T]\times K_\mx\times \S^{D-1}; L^2(K_\mlambda)),
\end{equation*} 
where $\nu$ is a positive and bounded Radon measure on 
$[0,T]\times K_\mx\times\S^{D-1}$, while $L^1_\nu$ denotes that 
integration is performed with respect to the measure $\nu$; 
otherwise, the standard Lebesgue measure is assumed.
The $\S^d$--projection of the measure $\nu$ 
is absolutely continuous with respect to the 
Lebesgue measure $d\mx\, dt$.  
\end{theorem}

\subsection{Stochastic velocity averaging}
Using the micro-local defect tools that 
we just introduced, we aim to establish the 
localization principle for the $H$-distribution. 
This will be done through the stochastic 
transport equations governing $u_n$. 
From the localization principle we will eventually 
reach the conclusion that the velocity averages of $u_n$, 
represented by $\actionb{u_n}{\rho}$ with 
$\rho \in C^{N_\mlambda}_c(\R^m)$, exhibit 
compactness in $L^1_{\loc}(\omega,t,\mx)$. 
In contrast to the deterministic case, the stochastic 
integral in \eqref{eq-s} amplifies the significance of the 
temporal variable $t$.  While the deterministic localization 
principle is derived by testing the kinetic 
PDEs against meticulously selected test functions, which involve 
specific multiplier operators acting on the $(t,\mx)$-variables, 
the stochastic context demands caution. This caution is 
essential to sidestep complications tied to the stochastic 
integrals and the temporal components of the (random) test functions. 
For instance, it is crucial to ensure that adaptivity remains intact. 
Besides, one would like to avoid computing explicitly the 
temporal Fourier transform of stochastic integrals. 
In \cite{Karlsen:2022aa}, these challenges were 
adeptly addressed by employing multiplier operators that act solely 
on the $\mx$-variable, thus forging a distinct strategy to 
infer strong temporal compactness.

The aim of our current study is develop a new strategy to address the 
intricacies introduced by the stochastic integrals. This enables us to 
employ micro-local defect tools that operate in both the temporal 
and spatial variables. This leads to a convergence 
result that apply under conditions that are considerably 
more lenient than those presented in \cite{Karlsen:2022aa}. 
Central to this approach is the representation 
of the stochastic forcing term as a distributional time derivative. 
This hinges on the $(t,\mx,\mlambda)$-weak 
formulation \eqref{weak-sense-txl-tmp1} 
of the kinetic SPDEs \eqref{eq-s}, 
diverging from the conventional It\^{o} formulation 
\eqref{weak-sense-xl-tmp1}.

Let $\Bbb{H}$ be a separable Hilbert space with 
norm $\norm{\cdot}_{\Bbb{H}}$ and inner product 
$(\cdot,\cdot)_{\Bbb{H}}$; the relevant 
example is \eqref{eq:H-space}. 
Given $r>1$, $\nu \in (0, 1)$, let $W^{\nu,r}(0,T;\Bbb{H})$ be 
the fractional Sobolev space of 
all $u \in L^r(0,T;\Bbb{H})$ such that
$$
\int_0^T \int_0^T
\frac{\norm{u(t)-u(s)}_{\Bbb{H}}^r}{\abs{t-s}^{1+\nu r}}
\, dt \, ds< \infty.
$$
This space is endowed with the norm
$$
\norm{u}_{W^{\nu,r}(0,T;\Bbb{H})}^r
= \int_0^T \norm{u(t)}_{\Bbb{H}}^r \, dt 
+ \int_0^T \int_0^T
\frac{\norm{u(t) - u(s)}_{\Bbb{H}}^r}{\abs{t-s}^{1+\nu r}} 
\, dt \, ds.
$$
We note in passing that 
the space $W^{\nu,r}(0,T;\Bbb{H})$ is continuously 
embedded into the H\"older space $C^\gamma([0,T];\Bbb{H})$ 
for all $\gamma \in (0,\nu-\frac{1}{r})$, provided that $\nu r>1$.  
The fractional Sobolev space 
$W^{\nu,r}(\cO)$, where $\cO\subset \R^{D}$ is an open
bounded set with Lipschitz boundary, 
is defined similary, see, e.g., \cite[page 17]{Grisvard:1985aa}. 
We recall that $W^{\nu_1,r}(\cO)$ is compactly 
embedded into $W^{\nu_2,r}(\cO)$ provided 
$1\ge \nu_1>\nu_2\ge 0$, see, e.g., 
\cite[Theorem 1.4.3.2]{Grisvard:1985aa}.

Given that $I_n(\Phi_n)$ denotes the stochastic integral 
\eqref{eq:stoch-int-n}, and according to \eqref{eq:stoch-int-n-deriv}, 
the temporal derivative $\pa_t I_n(\Phi_n)$ resides in 
the space $L^2(\Omega;W^{-1,\infty}([0,T];L^2_{\mx,\mlambda})\bigr)$. 
However, this is insufficient to derive the localization principle for the 
$H$-distribution.  Fortunately, we have further information that 
$I_n(\Phi_n)$ is uniformly bounded in 
$L^2\bigl(\Omega;W^{1/4,2}([0,T];L^2_{\mx,\mlambda})\bigr)$. 
As a result, $\pa_t I_n(\Phi_n)$ also exhibits a uniform bound in 
$L^2(\Omega;W^{-3/4,2}([0,T];L^2_{\mx,\mlambda})\bigr)$. 
This constraint suffices to derive the localization principle, consistent 
with the methodology of the deterministic proof.

According to \eqref{weak-sense-txl-tmp1}, 
for $\varphi\in \cD_{t,\mx}$ and 
$\varrho \in \cD_{\mlambda}$, the following equation holds a.s.:
\begin{equation}\label{weak-sense-txl}
	\begin{split}
		& \qquad \int_0^T \int_{\R^{d+m}} 
		u_n \pa_t (\varphi \varrho) 
		+u_n \mff(t,\mx,\mlambda)\cdot \nabla_\mx 
		(\varphi \varrho)
		\, d\mlambda\, d\mx \,dt
		\\ & \qquad\,\,
		= (-1)^{N_\mlambda+1} 
		\innb{g_n,\pa_{\mlambda}^{\malpha}
		(\varphi\varrho)}_{W^{-1,\alpha}_{t,\mx,\mlambda},
		W^{1,\alpha'}_{t,\mx,\mlambda}}
		+ \int_0^T\int_{\R^{d+m}} 
		I_n(\Phi_n)\pa_t (\varphi\varrho) 
		\, d\mlambda\, d\mx \, dt,
	\end{split}	
\end{equation} 
where $\alpha>1$ is given by \eqref{eq:intro-ass3-stoch} 
($\frac{1}{\alpha}+\frac{1}{\alpha'}=1$)  
and $\abs{\malpha}=N_\mlambda\ge 0$. 
Consider a random test function of the form 
$$
\En_Z(\omega) \varphi(t,\mx)\varrho(\mlambda), 
\quad 
\text{$Z\subset \Omega$ measurable, 
$\varphi \in \cD_{t,\mx}$, and 
$\varrho \in \cD_\mlambda$}.
$$ 
Starting with \eqref{weak-sense-txl}, we multiply by 
$\En_Z(\omega)$ and integrate over $\omega\in \Omega$. 
Notably, functions of the form $\En_Z\varphi\varrho$ are locally dense in 
$L^\infty\bigl(\Omega; \cD([0,T]\times \R^d;\cD(\R^m))\bigr)$
and consequently in $L^\infty\bigl(\Omega; W^{1,\tilde{p}}_c([0,T]\times \R^d;
C^{N_\mlambda}_c(\R^m))\bigr)$, 
where $\tilde{p}:=\max\{p',\alpha'\}$, $p$ is 
given by \eqref{eq:intro-ass1-stoch} and 
$\frac{1}{p}+\frac{1}{p'}=1$. With this understanding, it is not difficult to use 
\eqref{weak-sense-txl} to verify that the following equation holds:
\begin{equation}\label{weak-sense-txl-omega}
	\begin{split}
		&\int_{\Omega}\int_0^T \int_{\R^{d+m}} 
		u_n \pa_t \phi+u_n 
		\mff(t,\mx,\mlambda)\cdot \nabla_\mx \phi 
		\, d\mlambda \, d\mx \, dt \, dP(\omega)
		\\ & \qquad
		= (-1)^{N_\mlambda} 
		\int_{\Omega} \inn{g_n,\pa_{\mlambda}^{\malpha}
		\phi}_{W^{-1,\alpha}_{t,\mx,\mlambda},
		W^{1,\alpha'}_{t,\mx,\mlambda}}\, dP(\omega)
		\\ & \qquad\qquad
		+\int_{\Omega}\int_0^T\int_{\R^{d+m}} 
		 I_n(\Phi_n)\, \pa_t \phi\, d\mlambda\, d\mx \, dt\, dP(\omega),
	\end{split}	
\end{equation} 
for any random test function ($\tilde{p}=p' \vee \alpha'$)
\begin{equation}\label{eq:random-test}
	\phi = \phi(\omega,t,\mx,\mlambda)
	\in L^\infty\bigl(\Omega;
	W^{1,\tilde{p}}_c([0,T]\times \R^d;
	C^{N_\mlambda}_c(\R^m))\bigr).
\end{equation}
Note that $\phi$ is regular in $t$, and thus 
$d \phi=\pa_t \phi\, dt$ a.s.

The following theorem is the main result of this section:

\begin{theorem}\label{thm:stoch-velocity}
For every $n \in \mathbb{N}$, suppose $u_n$ satisfies 
the kinetic SPDE \eqref{weak-sense-txl-tmp1} and 
\begin{equation}\label{eq:un-weak-conv}
	u_n\in_b  L^p_{t,\mx,\mlambda}, \quad 
	\text{$u_n \weakn 0$ in $L^p_{t,\mx,\mlambda}$, a.s.}, 
	\quad p>1.
\end{equation} 
The SPDE \eqref{weak-sense-txl-tmp1} is driven by 
a deterministic drift $\mff$, a random source $g_n$, 
and a stochastic forcing term 
\eqref{eq:stoch-int-n}. Further, the stochastic forcing term is 
characterized by a noise coefficient $\Phi_n$ alongside 
a cylindrical Wiener process \eqref{eq:Wiener-process}, 
evolving over the stochastic basis \eqref{eq:stoch-basis}.
Suppose the drift $\mff$ meets the $L^q$ requirement 
\eqref{eq:intro-ass2} and the non-degeneracy 
condition \eqref{non-deg}. The source $g_n$ 
satisfies \eqref{eq:intro-ass3-stoch}, 
while the noise function $\Phi_n$ obeys 
\eqref{eq:intro-assPhi-stoch}. 

Then, a subsequence of $\seq{u_n}$ (which we do not relabel) 
exists, such that for all averaging functions 
$\rho \in C^{N_\mlambda}_c(\R^m)$,
\begin{equation*}
	\int_{\R^m} u_n(\cdot_\omega,\cdot_t,\cdot_\mx,
	\mlambda)\rho(\mlambda) \, d\mlambda 
	\ton 0 \quad 
	\text{in $L^1(\Omega\times [0,T]\times K)$, 
	$\forall K\Subset \R^d$}. 
\end{equation*}
\end{theorem}

\begin{remark}
Following Remark \ref{rem:non-zero-weak-limit}, we can, without 
loss of generality, assume that $u_n$ weakly converges to zero 
as indicated by \eqref{eq:un-weak-conv}. 
Should $u_n$ weakly converge to some non-zero $u$---in line with 
our original assumption \eqref{eq:intro-ass1-stoch}---then we 
will consider the modified variable $\tilde{u}_n := u_n - u$. 
This variable satisfies the SPDE
$$
\pa_t \tilde u_n+\Div_\mx 
\bigl (\, \mff(t,\mx,\mlambda) \tilde u_n\, \bigr)
=\pa_{\mlambda}^{\malpha} g_n
+\pa_t u+\Div_\mx \bigl (\, \mff(t,\mx,\mlambda) u\, \bigr)
+\pa_t I_n(\Phi_n),
$$
which align closely in form with \eqref{weak-sense-txl-tmp1}.
\end{remark}

\begin{proof}
The proof is largely analogous to that of 
Theorem \ref{thm:main-result-determ}. Our primary focus 
is to underscore the key differences between the deterministic and 
stochastic scenarios. For further specifics not covered here, 
consult the proof of Theorem \ref{thm:main-result-determ}.

Define $v_n:=V_n-V$, where $V_n$ is defined as in \eqref{vn}, but 
with $u_n$ substituted by the specific 
$u_n = u_n(\omega,t,\mx,\mlambda)$ referenced in the present theorem. 
Moreover, we choose $V$ such that $v_n$ weakly converges to 0 in 
$L^{p_v}(\Omega \times [0,T] \times \mathbb{R}^d)$ for 
any finite $p_v$.  As previously highlighted (see 
also \cite{Karlsen:2022aa}), by effectively using the 
Skorokhod-Jakubowski representation theorem 
\cite{Jakubowski:1997aa}, we may assume the almost sure weak 
convergence of $v_n$ with respect to the variables 
$(t,\mx,\mlambda)$, so that a.s.~$v_n\weakstar 0$ 
in $L^\infty([0,T]\times \R^d)$ and thus
\begin{equation}\label{eq:vn-weak-conv-stoch}
	v_n\weakn 0 
	\quad \text{in $L^{p_v}([0,T]\times \R^d)$, a.s., 
	for any $p_v<\infty$}.
\end{equation}
The convergence \eqref{eq:vn-weak-conv-stoch} 
here substitutes for 
\eqref{eq:vn-weak-conv} in the deterministic proof.

We will use a test function $\phi_n$ 
in \eqref{weak-sense-txl-omega} roughly 
of the form $(-\Delta_{t,\mx})^{-\frac12}
\bigl(\cA_{\overline{\psi}}(v_n)\bigr)$, 
where $\Delta_{t,\mx}$ is the $(t,\mx)$-Laplacian, so that 
the $(t,\mx)$-gradient $\nabla_{t,\mx} \phi_n$ 
contains the Riesz transform 
$\cR_{t,\mx}=\nabla_{t,\mx}(-\Delta_{t,\mx})^{-\frac12}$ 
of $\cA_{\overline{\psi}}(v_n)$. More precisely,
\begin{equation}\label{eq:test-func-random}
	\begin{split}
		&\phi_n(\omega,t,\mx,\mlambda)
		=\rho(\mlambda)\tilde{\rho}(\mlambda) 
		\ph(t,\mx)\tilde{\ph}(t,\mx) 
		\overline{\left(\cT_{-1}\circ 
		\cA_{\overline{\psi}\left(\mxi/\abs{\mxi}\right)}
		(v_n(\omega,\cdot))\right)(t,\mx)},
		\\ & 
		\quad \text{where} \quad 
		\psi \in C^{d+1}(\S^{d}), \quad 
		\ph, \tilde{\ph} \in C^1_c([0,T]\times \R^{d}), \quad 
		\rho, \tilde{\rho} \in C^{N_\mlambda}_c(\R^m).
	\end{split}
\end{equation}
We note that $\tilde{\ph}$ and $\tilde{\rho}$ are not 
arbitrary test functions but are specifically the 
ones used to define $v_n$.

For any $L^2_{t,\mx,\mlambda}$-valued random variable 
$v$, the corresponding random variable $\cA_{\overline{\psi}}(v)
=\cF^{-1}_{t,\mx}\bigl(\bar \psi \cF_{t,\mx}(v)\bigr)$ takes the form
$$
\cA_{\overline{\psi}}(v)(\omega,t,\mx)
=\int e^{2\pi i \left((t-\tau)\xi_0
+(\mx-\my)\cdot \mxi' \right)}\overline{\psi(\xi_0,\mxi')}
v(\omega,\tau,\my)
\, d\tau \, d\my \, d\xi_0 \, d\mxi',
$$ 
recalling that $\mxi=(\xi_0,\xi_1,\ldots,\xi_d)
=(\xi_0,\mxi')\in \R^{d+1}$. Note that
$$
\pa_t\cA_{\bar \psi}(v)(\omega,t,\mx)
=2\pi i  \int e^{2\pi i \left((t-\tau)\xi_0
+(\mx-\my)\cdot \mxi \right)}
\xi_0 \overline{\psi(\xi_0,\mxi)}
v(\omega,\tau,\my)\, d\tau \, d\my \, d\xi_0 \, d\mxi'.
$$
Thus, the process $(\omega,t)\mapsto 
\cA_{\bar \psi}(v)(\omega,t,\cdot)$  has zero quadratic variation. 
Therefore $\phi_n$ may serve as a random test function in 
accordance with \eqref{eq:random-test}. Notably, $\phi_n$ has 
compact support across all its variables (independently of $n$), 
and is $C^{N_\mlambda}$ in the $\mlambda$ variable. 
The object $\Psi_n := \cA_{\overline{\psi}}(v_n)$ 
remains uniformly bounded in $L^r_{\loc}$ for any $r\in [1, \infty)$, 
and in $L^\infty$ in $\omega$. This uniform boundedness is a 
result of $\psi$ being an appropriate Fourier multiplier 
and $v_n\in_b L^\infty_{\omega,t,\mx}$. 
Additionally, the Riesz potential $\cT_{-1}(\Psi_n)$ is uniformly 
bounded in $L^{\tilde r}_{\loc}$ for all $\tilde r\in [1, r^\star]$, 
where $r^\star > r$ is the Sobolev conjugate of $r$, 
see Theorem \ref{thm:riesz}. Lastly, the gradient 
$\nabla_{t,\mx}\cT_{-1}(\Psi_n) = -\cR_{t,\mx}(\Psi_n)$ 
also stays uniformly bounded in $L^r_{\loc}$ for any $r\in (1, \infty)$. 
This is assured by the fact that the Riesz transform 
$\cR_{t,\mx}$ is bounded on these spaces.

Therefore, we can utilize \eqref{eq:test-func-random} 
as a test function in \eqref{weak-sense-txl-omega}. 
In the resulting equation, we let $n\to \infty$ 
along the subsequence of $\seq{u_n}$ and $\seq{v_n}$ 
that define the $H$-distribution $B$, 
according to Theorem \ref{thm:H-distr-s}.
  
The substantial distinction from the deterministic 
proof \eqref{v2} is the limit 
of the term $\cI_n$ as $n\to \infty$, where
$$
\cI_n:=E\left[\int_0^T\int_{\R^{d+m}} 
\pa_t \phi_n I_n(\Phi_n)\, d\mlambda\, d\mx \, dt\right].
$$
We assert that $\cI_n\ton 0$. Let $r \ge 2$, $\nu < \frac{1}{2}$ be given. 
Suppose $\Phi_n$ is a progressively measurable process that 
resides in $L^r\bigl(\Omega \times [0, T];
L_2(\Bbb{K},L^2_{\mx,\mlambda})\bigr)$, where 
$L^2_{\mx,\mlambda}$ is short-hand for \eqref{eq:H-space}. 
Relying on \cite[Lemma 2.1]{Flandoli:1995aa} 
(with $H = L^2_{\mx,\mlambda}$), it follows that
$$
I_n(\Phi_n) \in L^r\bigl(\Omega; 
W^{\nu,r}([0,T];L^2_{\mx,\mlambda})\bigr)
$$
and there exists a constant $C(r,\nu) > 0$, 
independent of $n$, such that
$$
E\left[ \norm{I_n(\Phi_n)}_{W^{\nu,r}(0,T;L^2_{\mx,\mlambda})}^r
\right] 
\leq C(r,\nu) E\left[
\int_0^T \norm{\Phi_n(t)}_{L_2(\Bbb{K},L^2_{\mx,\mlambda})}^r \,dt
\right].
$$
In what follows, we specify $r=2$, in which case the right-hand 
side is uniformly bounded with 
respect to $n$, as stipulated by 
assumption \eqref{eq:intro-assPhi-stoch}. 
We can then estimate $\cI_n$ as follows:
\begin{align*}
	\cI_n^2 & \leq E\left[ 
	\norm{\pa_t \phi_n}_{W^{-\nu,2}([0,T];L^2_{\mx,\mlambda})}^{2} 
	\right]E\left[
	\norm{I_n(\Phi_n)}_{W^{\nu,2}([0,T];L^2_{\mx,\mlambda})}^2 
	\right]
	\\ & \lesssim 
	E\left[ \norm{\phi_n}_{W^{{1-\nu},2}([0,T]\times \R^{d+m})}^2\right]
	\qquad \text{(with $1-\nu>1/2$)}.
\end{align*}
Clearly, $\phi_n$ converges weakly to zero 
in $L^2$. Moreover, as discussed above, $\phi_n$ 
is uniformly bounded in $W^{1,2}
\Subset W^{\kappa,2}$, $\kappa\in [0,1)$. 
Therefore, it follows that $\phi_n \ton 0$ a.s.~in 
$W^{\kappa,2}([0,T]\times \R^{d+m})$ 
for any $\kappa<1$. Moreover, a.s., 
\begin{equation*}
	\norm{\phi_n}_{W^{{1-\nu},2}([0,T]\times \R^{d+m})}
	\leq \norm{\phi_n}_{W^{1,2}([0,T]\times \R^{d+m})}
	\lesssim_{T,K} \norm{v_n}_{L^\infty_{\omega,t,\mx}},
\end{equation*} 
where $K:=\supp(\phi_n)$. Thus, by 
Lebesgue's dominated convergence theorem, 
\begin{equation*}
	E\left[ \norm{\phi_n}_{W^{{1-\nu},2}([0,T]\times \R^{d+m})}^2 
	\right]\ton 0.
\end{equation*}  
This leads us to the desired result $\cI_n\ton 0$.

Lastly, we address the term associated with 
$g_n$. As argued near  \eqref{eq:ass-g-zero}, we can  
without loss of generality assume that the limit of $g_n$, 
see \eqref{eq:intro-ass3-stoch}, is zero: $g=0$. 
Similar to the deterministic case, we observe that the 
properties of $\partial_{\mlambda}^{\alpha}\phi_n$ mirror those of $\phi_n$.
Specifically, we then find that
\begin{equation*}
	\begin{split}
		& E\bigl[\left\langle g_n, 
		\pa_{\mlambda}^{\malpha}\phi_n \right\rangle\bigr] 
		\\ & \quad 
		\leq E\bigl[\norm{g_n}_{W^{-1,\alpha}([0,T]\times \R^{d+m})} 
		\norm{\pa_{\mlambda}^{\malpha}
		\phi_n}_{W^{1,\alpha'}([0,T]\times \R^{d+m})}\bigr] 
		\\ & \quad 
		\leq \norm{\pa_{\mlambda}^{\malpha}
		\phi_n}_{L^\infty(\Omega;W^{1,\alpha'}([0,T]\times \R^{d+m}))} 
		E\left[\norm{g_n}_{W^{-1,\alpha}([0,T]\times \R^{d+m})}
		\right]\ton 0,
	\end{split} 
\end{equation*} 
where we have used \eqref{eq:intro-ass3-stoch}. 

From here onwards, the argument closely follows the 
deterministic reasoning, leading us directly to the localization 
principle given by \eqref{v3}. The rest of the argument remains 
the same as before. With this, we conclude the proof.
\end{proof}

\appendix 

\section{$H$-measures and $H$-distributions}
\label{sec:appendix-H-distribution}

This appendix complements the core content of the paper by 
recalling relevant results from harmonic analysis and $H$-distributions. 
Engaging with this material is pivotal for a 
thorough grasp of the results presented.

For an introduction to harmonic analysis, we refer 
the reader to the sources \cite{Grafakos:2014aa,Stein:1970pr}. 
We will specifically recall the 
concept of a Fourier multiplier operator 
and the Marcinkiewicz multiplier theorem.  
A Fourier multiplier operator $\cA_\psi:L^2(\R^D)\to L^2(\R^D)$ 
associated to a function $\psi\in L^\infty(\R^D)$, $D\ge 1$, 
is a mapping given by $\cA_\psi(u)=\cF^{-1}
\bigl(\psi \cF(u)\bigr)$, where $\cF(u)(\mxi)=\hat{u}(\mxi)
=\int_{\R^d}e^{-2\pi i \mx\cdot
\mxi}\, u(\mx)\, d\mx$ is the Fourier transform of $u$, 
while $\cF^{-1}$ is the inverse Fourier transform. 

The following (Marcinkiewicz) multiplier theorem is from 
\cite[Corollary 6.2.5]{Grafakos:2014aa}.

\begin{theorem}\label{thm:multipliers}
Consider a function $\psi$ that is bounded on $\R^D$, for $D\ge 1$, 
and belongs to $C^D\bigl(\R^D\backslash
\bigcup_{j=1}^D \{\xi_j=0\}\bigr)$, 
for which there exists a constant $A>0$ such that for all 
multi-indices ${\malpha}=(\alpha_1,\dots,\alpha_D) 
\in \N_0^D$ with $\abs{{\malpha}}
=\alpha_1+\dots+\alpha_D\leq D$, the 
following inequality holds:
\begin{equation}\label{eq:multipliers}
	\abs{\mxi^{{\malpha}}\partial^{{\malpha}}\psi(\mxi)}
	\leq A, \quad \mxi\in \R^D \setminus
	\bigcup_{j=1}^D\seq{\xi_j=0}.
\end{equation}
Then $\psi$ is a multiplier in the $L^p$ 
space for all $p\in(1,\infty)$, and 
the operator norm of $\mathcal{A}_\psi$ is equal 
to $C_{D,p}\bigl( \norm{\psi}_{L^\infty}+A\bigr)$, where 
$C_{D,p}$ depends only on $D$ and $p$.  For $p=2$, 
the operator norm is $\norm{\psi}_{L^\infty}$. 
Furthermore, for any $\psi\in C^D(\bS^{D-1})$, the 
function $\psi\left(\frac{\mxi}{\abs{\mxi}}\right)$ satisfies 
the conditions of the theorem with $A=\|\psi\|_{C^D(\bS^{D-1})}$.
\end{theorem}

Background material on Riesz potentials 
$\cT_{-s}=(-\Delta)^{-\frac{s}{2}}$ and 
Riesz transforms $\cR=\nabla (-\Delta)^{-\frac{1}{2}}$ 
can be found in \cite{Stein:1970pr}. Here we recall 
the following result: 

\begin{theorem}\label{thm:riesz}
The Riesz potential, 
denoted $\cT_{-s}$ for $s\in (0,D)$, is a Fourier 
multiplier operator with the symbol $1/\abs{\mxi}^s$, 
where $\mxi \in \R^D$. It maps continuously from $L^p(\R^D)$, 
$p\in (1,D/s)$, to $L^{p^\star}(\R^D)$ 
for $p^\star=\frac{Dp}{D-sp}$. 
Moreover, the Fourier multiplier operator with the 
symbol $\xi_j^s/\abs{\mxi}^s$, $j=1,\dots,D$, maps 
continuously from $L^r(\R^D)$ 
to $L^{r}(\R^D)$ for any $r\in (1,\infty)$. 
\end{theorem}

As per the statement made in 
\cite[Theorem 2.2]{LazarMitrovic:16} (see 
also \cite[Theorem 16 and Corollary 18]{Erceg:2023aa}), 
there is an $H$-measure $\mu$ representing the 
distributional limit of products 
such as $u_n \cA_{\overline{\psi}}(v_n)$.  
We present the theorem in a format tailored 
for our specific application, in which case 
the functions in question are defined 
on either $\R^{d+1+m}$ or $\R^{d+1}$, with 
independent variables $t\in \R$, $\mx\in \R^d$, 
and $\mlambda\in \R^m$, so that below 
``$D=d+1$ and $\mx \to (t,\mx)$". Recall that 
for time-dependent functions defined on $t \geq 0$, 
we always extend them by zero for $t < 0$.  
Throughout this paper, we often say that a 
sequence $\seq{u_n}$ of functions 
is \textit{uniformly compactly supported}. 
This means that there exists a compact set 
$K$ such that each function $u_n$ has 
its support contained within $K$.

The next theorem combines results from 
\cite[Theorem 2.2]{LazarMitrovic:16}, 
\cite[Theorem 9]{Misur:2015aa}, and \cite{Gerard:91}. 
We first recall some notions from \cite{Gerard:91}. 
Let $H$ be a separable Hilbert space. We denote by 
$\mathcal L(H)$ the space of bounded linear 
operators on $H$, by $\mathcal L^1(H)$ the ideal 
of trace-class operators on $H$, and by 
${\cal K}(H)$ the space of compact operators on $H$. 
We recall that the dual ${\cal K}(H)'$ can be identified with 
${\cal L}^1(H)$ through the trace pairing. 
If $S\in\mathcal L^1(H)$, then its trace is denoted by 
$\tr(S)$. We shall use only the following standard 
properties: if $S\in\mathcal L^1(H)$ and 
$\tilde S\in\mathcal L(H)$, then 
$\tilde S S\in\mathcal L^1(H)$, 
the quantity $\tr(\tilde S S)$ is well defined, and
$$
|\tr(\tilde S S)|\leq 
\|\tilde S\|_{\mathcal L(H)}
\|S\|_{\mathcal L^1(H)} .
$$
For positive 
$S\in\mathcal L^1(H)$, one has
$$
\tr(S)=\sum_{k=1}^\infty
\langle Se_k,e_k\rangle_H
$$
for any orthonormal basis $(e_k)$ of $H$, 
and this value is independent 
of the choice of basis.

\begin{theorem}\label{thm:gen-H-measure}
Let $\seq{u_n}$ be bounded in $L^p(\R^{D+m})$, 
for some $p> 2$, and uniformly compactly supported in  
$$
K_\mx \times K_\mlambda \Subset \R^D\times \R^m \,.
$$ 
Let $\seq{v_n}$ be another sequence that is uniformly compactly 
supported in $K_\mx$ and weakly-$\star$ converges to 
zero in the space $L^\infty(\R^D)$. Let $r'>1$ be such 
that 
$$
\frac{1}{p}+\frac{1}{r'}< 1 \,.
$$

Then there exists a subsequence (not relabelled) 
and a continuous functional 
\begin{equation}\label{eq:mu-Bochner}
	\mu\in L^2_{w\star}
	\bigl(K_\mlambda ;\cM(K_\mx\times \S^{D-1})\bigr) 
	\cap L^{r}_{w\star}
	\bigl(K_\mx\times K_\mlambda; (C^D(\S^{D-1}))'\bigr)
\end{equation}
such that for every
$\varphi\in L^{r'}(K_\mx\times K_\mlambda)$ 
and $\psi\in{\rm C}^D(\S^{D-1})$, one has
\begin{equation}\label{mu-repr-h}
	\mu(\varphi \psi)
	=\lim_{n\to\infty}\int_{K_\mx\times K_\mlambda}
	\varphi(\mx,\mlambda) u_n(\mx,\mlambda)
	\overline{\cA_{\bar{\psi}(\mxi/|\mxi|)}(v_n)(\mx)}
	\,d\mlambda\, d\mx.
\end{equation}  We will refer to 
$\mu$ as the generalised $H$-measure (generated by 
the sequences $\seq{u_n}$ and $\seq{v_n}$). 
The generalised $H$-measure $\mu$ 
has the representation
\begin{equation}\label{eq:mu-repr-RN}
	\mu(\mx,\mxi,\mlambda)=
	g(\mx,\mxi,\mlambda)
	\, d\lambda \, \nu(\mx,\mxi), 
	\quad 
	g\in L_\nu^1(K_\mx\times\S^{D-1};L^2(K_\mlambda)),
\end{equation} 
where $\nu\in \cM(K_\mx\times\S^{D-1})$ 
is a positive bounded Radon measure on 
$K_\mx\times\S^{D-1}$, while $L^1_\nu$ denotes that 
integration is performed 
with respect to the measure $\nu$. 
Moreover, the $\S^{D-1}$--projection 
of the measure $\nu$ is absolutely 
continuous with respect to the 
Lebesgue measure $d\mx$ and, for 
a.e.~$\mx\in K_\mx$, there exists
a Radon probability measure $\nu_\mx$ 
and $h\in L^{\frac p2}(K_\mx)$
such that 
\begin{equation}\label{eq:nu-slicing}
	d\nu (\mx,\mxi) 
	=d\nu_\mx(\mxi)h(\mx)
	\, d\mx .
\end{equation}	
\end{theorem}

\begin{remark}\label{rem:weak-Bochner-spaces}
Recall that $\mu\in L^2_{w\star}(K_\mlambda;
\cM(K_\mx\times \S^{D-1}))$
means that $\mu$ is  a weak-$\star$ measurable function 
$\mu:K_\mlambda \to \cM(K_\mx\times \S^{D-1})$, that is, 
$\mlambda \mapsto \actionb{\mu(\mlambda)}{\phi}$
is measurable from $K_\mlambda$ into $\R$ 
for every $\phi \in C_0(K_\mx\times \S^{D-1})$, such that
$$
\int_{K_\mlambda} 
\norm{\mu(\mlambda)}^2_{\mathcal{M}(K_\mx\times\S^{D-1})} 
\, d\mlambda < \infty,
$$
see e.g.~\cite[p.~606]{Edwards:1965aa}.
Here $\cM(K_\mx\times\S^{D-1})$ is the 
dual of $C_0(K_\mx\times\S^{D-1})$
and corresponds to the space of 
bounded signed Radon measures on 
$K_\mx\times\S^{D-1}$.

On the other hand, 
$L^2_{w\star}(K_\mlambda ;\cM(K_\mx\times \S^{D-1}))$
is isomorphic to the dual of the Bochner space 
$L^2(K_\mlambda; C_0(K_\mx\times \S^{D-1}))$ 
(since $C_0(K_\mx\times\S^{D-1})$ lacks the 
Radon-Nikodym property).
Moreover, it is continuously embedded into 
the space of bounded linear operators
$\mathcal{L}(C_0(K_\mx\times\S^{D-1}); L^2(K_\mlambda))$
(it corresponds to the closed subspace 
of majorized linear operators; 
cf.~\cite[Theorem 2.1]{LazarMitrovic:13}).

Analogously, the space 
$L^{r'}_{w\star}\bigl(K_\mx\times K_\mlambda;
(C^D(\S^{D-1}))'\bigr)$ is isomorphic to the dual of 
$L^{r}(K_\mx\times K_\mlambda; C^D(\S^{D-1}))$, while its
elements are weak-$\star$ measurable functions 
from $K_\mlambda$ to $(C^D(\S^{D-1}))'$.
\end{remark}

\begin{proof}[Proof of Theorem \ref{thm:gen-H-measure}]
The existence of the generalised 
$H$-measure $\mu$ from the 
class \eqref{eq:mu-Bochner} along with 
\eqref{mu-repr-h} is precisely 
\cite[Theorem 2.2]{LazarMitrovic:16} 
and \cite[Theorem 9]{Misur:2015aa}.
We only prove the representation 
\eqref{eq:mu-repr-RN}--\eqref{eq:nu-slicing}.
A version of the representation formula 
for $\mu$ appears in 
\cite{Lazar:2012aa,LazarMitrovic:16}, but 
some steps of the derivation are not
spelled out there. We therefore 
provide a complete proof of
\eqref{eq:mu-repr-RN}--\eqref{eq:nu-slicing}.

Set
$$
{\bf K}:=K_{\mx}\times \S^{D-1},
\qquad
H:=L^2(K_{\mlambda};\C).
$$
Note that $H$, as well as all 
other function spaces in this paper, 
are complex vector spaces. For brevity, we generally omit 
the explicit reference to $\C$ in our notation 
(e.g., writing $L^2(K_\mlambda)$ instead of 
$L^2(K_\mlambda; \C)$).

\medskip
\paragraph{{\it Step 1 (G\'erard's microlocal defect measure)}}
Set
$$
U_n(\mx):=u_n(\mx,\cdot)\in H,
\qquad
V_n(\mx):=v_n(\mx)\En_{K_\mlambda}(\cdot)\in H,
$$
so that
\begin{equation}\label{eq:Wn-def}
	W_n(\mx):=\bigl(U_n(\mx),V_n(\mx)\bigr)^\top
	\in H\times H.
\end{equation}
Since $p>2$ and the sequence $\seq{u_n}$ 
is uniformly supported in
$K_{\mx}\times K_{\mlambda}$, the sequence 
$(U_n)$ is bounded in $L^2(K_\mx;H)\cap L^p(K_\mx; H)$. 
Moreover, since $\seq{v_n}$ is uniformly supported in
$K_{\mx}$ and converges weakly-$\star$ 
to $0$ in $L^\infty(\R^D)$, the sequence $\seq{V_n}$ 
converges weakly to $0$ in $L^q(K_\mx;H)$
for any $q\in [1,\infty)$
(recall that $K_\mlambda$ is bounded).
Passing to a subsequence, we may assume 
that $U_n\rightharpoonup U$ weakly
in $L^2(K_\mx;H)$. Since 
$V_n\rightharpoonup0$ in $L^2(K_\mx;H)$, replacing
$U_n$ by $U_n-U$ does not affect the mixed 
limit in \eqref{mu-repr-h}, and we may therefore
assume that
$$
W_n\rightharpoonup0
\qquad\text{weakly in }L^2(K_\mx;H\times H).
$$

By G\'erard's construction \cite[Theorem 1]{Gerard:91}, 
after extraction of a subsequence (still denoted
by $\seq{W_n}_{n\in\N}$),
there exists a $2\times 2$ 
block matrix of operator-valued measures
$$
M\in \cM_+\bigl({\bf K};
\mathcal L^1(H\times H)\bigr), \quad
M=
\begin{pmatrix}
	M_{1,1} & M_{1,2}\\
	M_{1,2}^* & M_{2,2}
\end{pmatrix},
$$
such that for every zero-order 
pseudodifferential operator
$A\in \Psi^0(\R^D;{\cal K}(H\times H))$, 
with principal symbol
$a(\mx,\mxi)$, one has
\begin{equation}\label{eq:gerard-def}
	\lim_{n\to\infty}
	\bigl(AW_n,W_n\bigr)_{L^2_{\mx}(H\times H)}
	= \int_{{\bf K}}
	\tr\bigl(a(\mx,\mxi)\,d M(\mx,\mxi)\bigr).
\end{equation}

Let
$$
\nu:=\tr M\in \cM_+({\bf K}).
$$
Then, by \cite[Proposition A.1]{Gerard:91},
$$
M=F\,\nu
$$
for some $\nu$-measurable $2\times2$ 
block field of trace-class operators
\begin{align*}
	& F\in L^1_\nu\bigl({\bf K};
	\mathcal L^1(H\times H)\bigr),
	\\ &
	F(\mzeta)\ge 0 
	\quad \hbox{and} \quad \tr F(\mzeta)=1
	\quad\text{for }\nu\text{-a.e. }
	\mzeta:=(\mx,\mxi)\in{\bf K}.
\end{align*}

\begin{remark}
Since the sequence $\seq{W_n}$ is bounded in $L^2$,
rather than merely in $L^2_{\loc}$, the
vector-valued measure $M$ and the scalar measure
$\nu$ are both bounded. Consequently,
$F \in L^1$, rather than only $L^1_{\loc}$.
This is the main difference from the results
in \cite{Gerard:91}.
\end{remark}

\medskip

\paragraph{{\it Step 2 (Slicing of $\nu$)}}
Let $\pi_\mx : K_\mx\times\S^{D-1} \to K_\mx$
denotes the canonical projection 
$\pi_\mx(\mx,\mxi)=\mx$.
Then we want to show that the pushforward measure
$\tilde\nu := (\pi_\mx)_\# \nu$
is absolutely continuous with respect 
to the Lebesgue measure $d\mx$.

Let $(e_k)_{k\ge1}$ be an orthonormal basis of $H$, 
and let $\mathsf P_N$ denote the orthogonal 
projection of $H$ onto
$\mathrm{span}\{e_1,\dots,e_N\}$.
Since each $\mathsf P_N$ is a finite-rank operator, 
it follows that $\mathsf{P}_N\in \mathcal{K}(H)$
(i.e., $\mathsf{P}_N$ is compact).
Moreover, $\mathsf{P}_N$ converges to the identity operator
in the strong operator topology.
Consequently, for any trace-class operator
$S\in\mathcal{L}^1(H)$, the sequence 
$\tr(\mathsf{P}_N S)$ converges to $\tr(S)$
as $N$ tends to infinity.
Since we work on $H\times H$, we introduce 
$$
\mathsf R_N:=
\begin{pmatrix}
	\mathsf P_N & 0\\
	0 & \mathsf P_N
\end{pmatrix},
$$
for which we still have that for any 
$S\in \mathcal{L}^1(H\times H)$,
the sequence $\tr(\mathsf{R}_N S)$ 
converges to $\tr(S)$.

Let us take an arbitrary 
$\varphi\in C^\infty_c(K_\mx)$. 
Then we have
\begin{align*}
	\int_{K_\mx} \varphi(\mx) \, d\tilde\nu(\mx)
	&= \int_{K_\mx}\int_{\S^{D-1}}
	\varphi(\mx) \,d\nu(\mx,\mxi) \\
	&= \int_{\mathbf{K}} \varphi(\mx) 
	\tr(F(\mzeta)) \, d\nu(\mzeta)
	= \lim_{N\to\infty} \int_{\mathbf{K}}
	\varphi(\mx) \tr(\mathsf{R}_N F(\mzeta)) 
	\, d\nu(\mzeta),
\end{align*}
where we have used that $\tr(F(\mzeta))=1$ 
for $\nu$-a.e.~$\mzeta\in\mathbf{K}$
and the previous discussion on $\mathsf{R}_N$. 

Since the operator with the symbol 
$\varphi(\mx)\mathsf{R}_N$
belongs to $\Psi^0(\R^D;{\cal K}(H\times H))$, 
the identity \eqref{eq:gerard-def} is applicable. 
Hence, we get
\begin{align}\label{eq:slicing}
	\int_{K_\mx} \varphi(\mx) \, d\tilde\nu(\mx)
	= \lim_{N\to\infty} \lim_{n\to\infty} \int_{K_\mx} \varphi(\mx) 
	\bigl(\mathsf{R}_N W_n(\mx),W_n(\mx)\bigr)_{H\times H} \,d\mx.
\end{align}
However, the expression on the 
right-hand side is a bounded linear functional 
on $L^q(K_\mx)$, where $q=(\frac{p}{2})'\in (1,\infty)$. 
Here we have used that $(W_n)$ is bounded in $L^p(K_\mx;H\times H)$. 
Since $(L^q(K_\mx))'$ is isomorphic to $L^{p/2}(K_\mx)$, there exists 
$h\in L^{p/2}(K_\mx)$ such that $d\tilde\nu(\mx) = h(\mx)d\mx$.

By applying the result on slicing measures 
\cite[Theorem 1.5.1]{Evans:1990wt} we finally get that
$$
d\nu(\mx,\mxi) = d\nu_\mx(\mxi) h(\mx) d\mx,
$$
where, for a.e.~$\mx\in K_\mx$, $\nu_\mx$ 
is a Radon probability measure 
on $\S^{D-1}$.

\begin{remark}
Note that, in general, the limits on the right-hand side 
of \eqref{eq:slicing} do not commute. Consequently, 
we generally have only $\tilde\nu \leq \delta$, 
where $\delta$ denotes the (defect) measure 
obtained as the weak-$\star$ limit of 
$\mathbf{x} \mapsto (W_n(\mathbf{x}), W_n(\mathbf{x}))_H$. 
Naturally, if $H$ were finite-dimensional, 
the measures $\tilde\nu$ and $\delta$ would coincide.
\end{remark}

\medskip

\paragraph{{\it Step 3 (Integral representation of $F$)}}
Since $H=L^2(K_\mlambda)$, any trace-class operator
$S\in\mathcal{L}^1(H\times H)$ is an integral operator
with a kernel 
$$
G^S:=
\begin{pmatrix}
	G^S_{1,1} & G^S_{1,2}\\
	G^S_{2,1} & G^S_{2,2}
\end{pmatrix} ,
$$
where $G^S_{i,j}\in L^2(K_\mlambda\times K_\mlambda)$,
i.e., for any $\Phi\in H\times H$ it holds
$$
(S\Phi)(\mlambda_1) 
= \int_{K_\mlambda} G^S(\mlambda_1,\mlambda_2) 
\Phi(\mlambda_2) \,d\mlambda_2 \,.
$$
The mapping 
$$
\mathcal{L}^1(H\times H)
\ni S\overset{\mathcal{J}}{\mapsto} G^S \in 
L^2(K_\mlambda\times K_\mlambda;\C^{2\times 2})
$$ 
is linear, bounded
and injective
(note that the preceding discussion holds for the broader class of 
Hilbert-Schmidt operators where $\mathcal{J}$ is an isomorphism; 
cf.~\cite[Chapter 12]{Aubin:2000aa}).
Moreover, since $S$ is a trace-class operator, 
by \cite{Brislawn:1988aa} one can choose the kernel $G^S$ such that 
$$
\tr (S) = \int_{K_\mlambda} 
\tr_{\C^{2\times 2}} \bigl(G^S(\mlambda,\mlambda)\bigr)
\,d\mlambda 
= \int_{K_\mlambda} \Bigl(G^S_{1,1}(\mlambda,\mlambda) 
+ G^S_{2,2}(\mlambda,\mlambda)\Bigr)
\,d\mlambda,
$$
where we denoted by $\tr_{\C^{2\times 2}}$ 
the trace of $2\times 2$ complex matrices.

For $\nu$-a.e.~$\mzeta\in\mathbf{K}$ we 
have $F(\mzeta)\in \mathcal{L}^1(H\times H)$.
Let us denote by $G^{\mzeta}$ a kernel of $F(\mzeta)$.
Since $F\in L^1_\nu(\mathbf{K};
\mathcal{L}^1(H\times H))$ and the mapping 
$\mathcal{J}$ is bounded, the function 
$G(\mzeta,\cdot):=G^\mzeta$, $\nu$-a.e.~$\mzeta\in\mathbf{K}$,
is strongly measurable with respect to all variables and 
$G\in L^1_\nu(\mathbf{K};L^2(K_\mlambda\times K_\mlambda))$
(see page 15 and Proposition 1.2.25 
in \cite{Hytonen:2016aa}).

\medskip

\paragraph{{\it Step 4 (Connection with \eqref{mu-repr-h})}}

Let us take $\rho_1,\rho_2\in H$ and define
$\Phi_1:=(0,\rho_1)\in H\times H$, 
$\Phi_2=(\bar\rho_2,0)\in H\times H$.
The kernel of the rank-one operator 
$\Phi_1\otimes \Phi_2
=(\cdot\,,\Phi_2)_{H\times H}\Phi_1$ is
given by (we take the inner product 
to be antilinear in the second argument):  
$$
G^{\Phi_1\otimes\Phi_2}(\mlambda_1,\mlambda_2)=
\begin{pmatrix}
	0 & 0\\
	\rho_1(\mlambda_1)\rho_2(\mlambda_2) & 0
\end{pmatrix}.
$$
Hence, for $\nu$-a.e.~$\mzeta\in\mathbf{K}$, 
it holds (see \cite[Proposition 3.3]{Brislawn:1988aa})
\begin{align*}
	\tr\bigl((\Phi_1\otimes \Phi_2) F(\mzeta)\bigr)
	&= \int_{K_\mlambda}\int_{K_\mlambda}
	\tr_{\C^{2\times 2}}
	\bigl(G^{\Phi_1\otimes\Phi_2}(\mlambda_1,\mlambda_2)
	G(\mzeta,\mlambda_2,\mlambda_1)\bigr) 
	\, d\mlambda_2 d\mlambda_1 \\
	&= \int_{K_\mlambda}\int_{K_\mlambda} 
	\rho_1(\mlambda_1)\rho_2(\mlambda_2)
	G_{1,2}(\mzeta,\mlambda_2,\mlambda_1) 
	\, d\mlambda_2 d\mlambda_1 \\
	&= \int_{K_\mlambda} \rho_2(\mlambda) 
	G_{1,2}^{\rho_1}(\mzeta,\mlambda) \, d\mlambda,
\end{align*}
where 
$$
G_{1,2}^{\rho_1}(\mzeta,\mlambda) 
:= \int_{K_\mlambda} \rho_1(\mlambda_1) 
G_{1,2}(\mzeta,\mlambda,\mlambda_1) \,d\mlambda_1.
$$

Applying \eqref{eq:gerard-def} to the 
operator $A$ with the symbol 
$(\mx,\mxi)\mapsto \varphi(\mx)
\psi(\frac{\mxi}{|\mxi|}) \Phi_1\otimes\Phi_2$,
for $\varphi\in C^\infty_c(K_\mx)$ 
and $\psi\in C^\infty(\S^{D-1})$, 
we get 
\begin{align*}
	\int\limits_{K_\mx}\int\limits_{\S^{D-1}} \int\limits_{K_\mlambda}
	\varphi(\mx) & \psi\left(\frac{\mxi}{|\mxi|}\right) \rho_2(\mlambda)
	G_{1,2}^{\rho_1}(\mx,\mxi,\mlambda) \, d\mlambda d\nu(\mx,\mxi) \\
	&= \int\limits_{K_\mx}\int\limits_{\S^{D-1}} 
	\varphi(\mx) \psi\left(\frac{\mxi}{|\mxi|}\right)
	\tr\bigl((\Phi_1\otimes \Phi_2) F(\mx,\mxi)\bigr) \, d\nu(\mx,\mxi) \\
	&= \int\limits_{K_\mx}\int\limits_{\S^{D-1}} 
	\tr\left(\varphi(\mx) \psi\left(\frac{\mxi}{|\mxi|}\right)
	(\Phi_1\otimes \Phi_2) \, d M(\mx,\mxi)\right) \\
	&= \lim_{n\to\infty} \bigl(AW_n,W_n\bigr)_{L^2_{\mx}(H\times H)}.
\end{align*}
Since, for $\mx\in K_\mx$, 
\begin{align*}
	(AW_n)(\mx) &= \mathcal{A}_{\psi(\mxi/|\mxi|)}
	\Bigl(\varphi(\cdot) (W_n(\cdot),
	\Phi_2)_{H\times H}\Bigr)(\mx)\Phi_1\\
	&= \mathcal{A}_{\psi(\mxi/|\mxi|)}
	\Bigl(\varphi(\cdot) (u_n(\cdot),
	\bar\rho_2)_{H}\Bigr)(\mx)\Phi_1,
\end{align*}
we have
\begin{align*}
\bigl(AW_n, & W_n\bigr)_{L^2_{\mx}(H\times H)} \\
	&= \int\limits_{K_\mx} \int\limits_{K_\mlambda}
	\mathcal{A}_{\psi(\mxi/|\mxi|)}
	\biggl(\int_{K_\mlambda}
	\!\!\varphi(\cdot)\rho_2(\mlambda) u_n(\cdot,\mlambda) 
	\,d\mlambda\biggr)(\mx) \,
	\rho_1(\mlambda_1) \overline{v_n(\mx)} \, d\mlambda_1 d\mx \\
	&= C_{\rho_1} \int\limits_{K_\mx}
	\int\limits_{K_\mlambda} 
	\varphi(\mx) \rho_2(\mlambda)
	u_n(\mx,\mlambda)
	\overline{\mathcal{A}_{\bar\psi(\mxi/|\mxi|)}(v_n)(\mx)} 
	\,d\mlambda d\mx,
\end{align*}
where $C_{\rho_1}=\int_{K_\mlambda} 
\rho_1(\mlambda)\,d\mlambda$.

By combining these identities 
with \eqref{mu-repr-h} we finally obtain
\begin{equation*}
	C_{\rho_1} \mu(\varphi\rho_2\psi) 
	= \int\limits_{K_\mx}\int\limits_{\S^{D-1}}
	\int\limits_{K_\mlambda}
	\varphi(\mx) \psi\left(\frac{\mxi}{|\mxi|}\right)
	\rho_2(\mlambda)G_{1,2}^{\rho_1}(\mx,\mxi,\mlambda) 
	\, d\mlambda d\nu(\mx,\mxi).
\end{equation*}
If we take $\rho_1 \equiv 1$ on $K_\mlambda$,
then $C_{\rho_1} = \meas(K_\mlambda) > 0$.
Define
$$
g(\mx,\mxi,\mlambda) :=
\frac{1}{C_{\rho_1}}
G_{1,2}^{\rho_1}(\mx,\mxi,\mlambda).
$$
Then the representation formula
\eqref{eq:mu-repr-RN} holds when tested
against functions of the form
$\varphi \rho_2 \psi$, where
$\varphi \in C_c^\infty(K_\mx)$,
$\psi \in C^D(\S^{D-1})$, and
$\rho_2 \in L^2(K_\mlambda)$ are arbitrary.
Since the algebraic tensor product
$C_c^\infty(K_\mx)\otimes
C^{D}(\S^{D-1})\otimes
L^2(K_\mlambda)$
is dense in the Bochner space
$L^2(K_\mlambda;
C_0(K_\mx\times\S^{D-1}))$, 
the claim follows.
\end{proof}

The preceding theorem does not meet our requirements 
because, in the context of Theorem \ref{thm:main-result-determ}, 
the sequence $\seq{u_n}$ only converges weakly in $L^p$ for a 
potentially small value of $p>1$.  The following theorem 
(see \cite[Theorem 2.2]{LazarMitrovic:13} 
or \cite[Theorem 1.1]{Lazar:2017aa}) 
provides a precise statement regarding the existence of 
the broader concept of $H$-distributions, denoted as $B$ to 
distinguish it from the generalized $H$-measure 
$\mu$ of Theorem \ref{thm:gen-H-measure}.

\begin{theorem}\label{thm:H-distr}
Let $\seq{u_n}$ be bounded in $L^p(\R^{D+m})$, 
for some $p>1$, and uniformly compactly supported in  
$K_\mx \times K_\mlambda \Subset \R^D\times \R^m$. 
Let $\seq{v_n}$ be another sequence that is uniformly compactly 
supported in $K_\mx$ and weakly-$\star$ converges to 
zero in the space $L^\infty(\R^D)$. Let $r'>1$ be 
such that $\frac{1}{p}+\frac{1}{r'}<1$. Then, by passing to a 
subsequence (not relabelled), there exists 
a continuous bilinear functional $B$ defined on 
$L^{r'}(\R^{D+m}) \otimes C^D(\S^{D-1})$, such that 
for any $\phi \in L^{r'}(\R^{D+m})$ 
and $\psi \in C^D(\S^{D-1})$,
\begin{equation}\label{rev1}
	B\bigl(\phi, \psi\bigr)=\lim\limits_{n\to \infty}
	\int_{\R^{D+m}} \phi(\mx,\mlambda)u_n(\mx,\mlambda) 
	\overline{\cA_{\overline{\psi}(\mxi/\abs{\mxi})} (v_n)(\mx)} 
	\, d\mlambda\,d\mx,
\end{equation}
where $\cA_{\overline{\psi}(\mxi/\abs{\mxi})}$ is the Fourier 
multiplier operator on $\R^D$ with 
symbol $\psi(\mxi/\abs{\mxi})$. 

The continuous bilinear functional $B$ will be referred to 
as the $H$-distribution (generated by the 
sequences $\seq{u_n}$ and $\seq{v_n}$).
\end{theorem} 

The next theorem expands the scope of the 
$H$-distribution $B$ from \eqref{rev1}, initially 
defined on the tensor product space $L^{r'}(\R^{D+m}) 
\otimes C^{D}(\S^{D-1})$, by extending it to a 
continuous linear functional on the Bochner space 
$L^{r'}(\R^{D+m};C^D(\S^{D-1}))$. The proof can be 
found in \cite[Corollary 2.3]{LazarMitrovic:13}
or \cite[Proposition 6]{Misur:2015aa}.

\begin{theorem}\label{thm:H-distr-extend}
The bilinear functional $B:L^{r'}(\R^{D+m}) \times C^D(\S^{D-1})\to \R$,
which is defined by \eqref{rev1}, can be extended to a 
continuous linear functional acting on the Bochner 
space $L^{r'}(\R^{D+m};C^D(\S^{D-1}))$. 
\end{theorem}

In other 
words, the $H$-distribution $B$ is an element of 
$L^r_{w\star}\bigl(\R^{D+m};(C^D(\S^{D-1}))'\bigr)$
(see Remark \ref{rem:weak-Bochner-spaces}).

The next lemma is a demonstration of the non-degeneracy 
\eqref{non-deg} of the drift vector $\mff=\mff(t,\mx,\mlambda)$. 
The proof is a conventional one, referring to, for 
instance, \cite[Theorem 3.4]{LazarMitrovic:16} for details.   
The lemma in question utilises an object denoted 
by $\mu$, which is the generalised $H$-measure 
of Theorem \ref{thm:gen-H-measure}, generated by a 
uniformly compactly supported sequence $\seq{u_n}$ 
that is bounded in $L^p(\mathbb{R}^{d+1+m})$, for some 
large $p$, and a uniformly compactly supported 
sequence  $\seq{v_n}$ that weakly-$\star$ converges 
to zero in $L^\infty(\mathbb{R}^{d+1})$. 
We use $\mxi=(\mxi_0,\mxi')\in \R^{d+1}$ 
to denote the dual (Fourier) variable of $(t,\mx)$.

\begin{lemma}\label{loc-p}
	Assuming that $\mff=\mff(t,\mx,\mlambda)$ is a vector-field that satisfies 
	\eqref{eq:intro-ass2} and the non-degeneracy condition \eqref{non-deg}, and 
	letting $\mu$ denote the generalized $H$-measure of 
	Theorem \ref{thm:gen-H-measure}, we have the following result: for any 
	$K\times L\Subset \R^{d+1}\times \R^m$ and 
	$\phi \in L^\infty(K\times L)\otimes C^{d+1}(\S^d)$,
	\begin{equation}\label{assum-11}
	\lim\limits_{\delta\to 0} \action{\frac{\delta \abs{\mxi'}^2 \phi}
		{\abs{\xi_0+\mff \cdot \mxi'}^2
			+\delta  \abs{\mxi'}^2}}{\mu} = 0.
	\end{equation}
\end{lemma}

\begin{proof}
	Using \eqref{eq:mu-repr-RN} and \eqref{eq:slicing}, 
	we can rewrite the expression 
	inside the limit in \eqref{assum-11} as follows:
	\begin{equation}\label{eq:limit-tmp1}
	\action{\frac{\delta \abs{\mxi'}^2 \phi}
		{\abs{\xi_0+\mff \cdot \mxi'}^2
			+\delta  \abs{\mxi'}^2}}{\mu}
	=\int_{K\times \S^d}
	I_\delta(t,\mx,\mxi) \, d\nu(t,\mx,\mxi),
	\end{equation}
	where 
	$$
	I_\delta(t,\mx,\mxi)
	=\int_{L} \frac{\delta \abs{\mxi'}^2 \phi}
	{\abs{\xi_0+\mff \cdot \mxi'}^2+\delta  \abs{\mxi'}^2} 
	\, g(t,\mx,\mlambda,\mxi) \, d\mlambda.
	$$
	According to the non-degeneracy condition \eqref{non-deg}, for 
	a.e.~$(t,\mx)\in K$, every $\mxi\in \S^d$, 
	and a.e.~$\mlambda\in L$,
	$$
	\frac{\delta \abs{\mxi'}^2 \phi}
	{\abs{\xi_0+\mff(t,\mx,\mlambda) \cdot \mxi'}^2
		+\delta  \abs{\mxi'}^2}\todelta 0. 
	$$ 
	Thus, by the dominated convergence theorem, 
	we conclude that $I_\delta(t,\mx,\mxi)\todelta 0$ 
	for a.e.~$(t,\mx)$ and every $\mxi$. 
	
	Since $\nu$ is absolutely continuous with respect to 
	the Lebesgue measure $d\mx\, dt$, 
	$I_\delta(t,\mx,\mxi)\todelta 0$ 
	for $\nu$-a.e.~$(t,\mx,\mxi)$. It is also clear that
	$$
	\abs{I_\delta(t,\mx,\mxi)}
	\leq \int_{L}  \abs{\phi(t,\mx,\mlambda,\mxi)}
	\abs{g(t,\mx,\mlambda,\mxi)} \, d\mlambda
	\in L_\nu^1(\R^{d+1} \times \S^d).
	$$
	Another application of the dominated 
	convergence theorem then results in
	$$
	I_\delta\todelta 0 \quad 
	\text{in $L_\nu^1(\R^{d+1} \times \S^d)$}.
	$$
	By combining this with \eqref{eq:limit-tmp1}, we obtain \eqref{assum-11}. 
	For related details, see \cite[p.~254]{Lazar:2012aa}.
\end{proof}

We need the following technical lemma, which shows that 
the Fourier multiplier operators 
having a symbol of the form 
$\frac{\delta \abs{\mxi'}^2}{\abs{\xi_0+F(\mlambda)\cdot \mxi'}^2
+\delta \abs{\mxi'}^2}$, for $\delta>0$, are bounded 
on $L^p$ spaces with $p>1$, as 
long as $F$ is locally integrable. 

\begin{lemma}\label{Lp-m}
For $F=F(\mlambda) \in L^1_{\loc}(\R^m;\R^d)$ 
and $\delta>0$, set 
$$
\psi_{F,\delta}=\psi_{F,\delta}(\mxi;\mlambda)
:=\frac{\delta \abs{\mxi'}^2}
{\abs{\xi_0+F(\mlambda)\cdot \mxi'}^2
+\delta \abs{\mxi'}^2}.
$$ 
Then, for any $p\in (1,\infty)$, the Fourier multiplier operator 
$\cA_{\psi_{F,\delta}}:L^p(\R^{d+1}) \to L^p(\R^{d+1})$
is bounded independently of $\mlambda$, $\delta$ and $F$. 
\end{lemma}

\begin{proof}
We leverage the fact that a linear 
transformation preserves the $L^p\to L^p$ norm of 
a Fourier multiplier operator, i.e., for any invertible 
matrix $M$, we have that $\norm{\cA_\psi}_{L^p\to L^p}
=\norm{\cA_{\psi(M^{-1}\cdot)}}_{L^p\to L^p}$,  
see e.g.~\cite[Lemma 12]{Erceg:2023aa}.
Thus, we introduce the change of variables 
$\sigma=\xi_0+F(\mlambda)\cdot \mxi'$ 
and $\meta=\sqrt{\delta}\mxi'$, and 
observe that it suffices to analyse the norm of the 
multiplier operator $\cA_{\Psi}$ with 
$\Psi:=\frac{ \abs{\meta}^2}{\sigma^2+\abs{\meta}^2}$. 

By applying the Marcinkiewicz multiplier 
theorem (Theorem \ref{thm:multipliers}), 
noting that $\R^{d+1}\setminus \seq{0}
\ni(\sigma,\eta)\mapsto \frac{\abs{\meta}^2}{\sigma^2+\abs{\meta}^2}$ 
is an admissible symbol, it is easy to show that 
$\norm{\cA_{\Psi}}_{L^p\to L^p} \leq C$, 
where the constant $C=C(d,p)$ does not depend on 
$\delta$, $F$.
\end{proof}

To establish Theorem \ref{thm:main-result-determ}, we rely 
on the following technical lemma, which builds upon the 
previous lemma and is used in a crucial step of the proof.

\begin{lemma}\label{crucial} 
Consider the characteristic 
functions $\chi_j=\chi_j(t,\mx)$, for $j=1,\ldots,N$,  
defined in \eqref{eq:f-approx}. 
Suppose that $h$ belongs to $L^p(\R^{d+1+m})$ 
and $v$ belongs to $L^{p'}(\R^{d+1})$, with $p\in (1,\infty)$ and 
$\frac{1}{p}+\frac{1}{p'}=1$, and both $h$ 
and $v$ have compact supports. Consider $N$ 
vector-functions $F_1,\ldots,F_N\in L^1_{\loc}(\R^m;\R^d)$. 
Fix  $\rho\in C_c(\R^m)$ and $\delta>0$. Then
\begin{align}\nonumber
	&\abs{\int_{\R^{d+m+1}} \sum_{j=1}^N 
	\rho(\mlambda)h(t,\mx,\mlambda)\chi_j(\mx)\,
	\cA_{\frac{\delta \abs{\mxi'}^2}
	{\abs{\xi_0+F_j(\mlambda)\cdot \mxi'}^2
	+\delta \abs{\mxi'}^2}}\left(v \chi_j\right)(t,\mx) 
	\, d\mlambda\, d\mx \, dt} 
	\\ & \qquad 
	\leq C \norm{h}_{L^{p}(\R^{d+1+m})} 
	\norm{\rho v}_{L^{p'}(\R^{d+1+m})},
	\label{cruc1}
\end{align} 
for a constant $C$ independent of $N$, 
$\delta$, and $F_1,\ldots,F_N$.
\end{lemma}

\begin{proof}
By applying the discrete H{\"o}lder inequality followed 
by the integral variant, 
\begin{align*}
	& \abs{\int_{\R^{d+m+1}} 
	\sum_{j=1}^N \rho(\mlambda)h(t,\mx,\mlambda)\,
	\chi_j(\mx) \cA_{\frac{\delta \abs{\mxi'}^2}
	{\abs{\xi_0+F_j(\mlambda)\cdot \mxi'}^2+\delta \abs{\mxi'}^2}}
	\left(v \chi_j\right)(t,\mx) d\mlambda\, d\mx \,dt}
	\\& \quad 
	\leq \int_{\R^{d+m+1}} 
	\left(\sum\limits_{j=1}^N \abs{h(t, \mx,\mlambda)}^p 
	\chi_j(t,\mx)\right)^{1/p} 
	\\ & \qquad\qquad \times
	\left(\sum_{j=1}^N \abs{\cA_{\frac{\delta \abs{\mxi'}^2}
	{\abs{\xi_0+F_j(\mlambda)\cdot \mxi'}^2+\delta \abs{\mxi'}^2}}
	\left(\rho v \chi_j\right)(t,\mx)}^{p'} \right)^{1/p'} 
	\, d\mlambda \, d\mx \, dt
	\\ & \quad
	\leq \norm{\sum_{j=1}^N h \chi_j}_{L^p(\R^{d+m+1})}  
	\\ & \qquad\qquad \times
	\left(\,\int_{\R^{d+1+m}} \sum_{j=1}^N 
	\abs{\cA_{\frac{\delta \abs{\mxi'}^2}
	{\abs{\xi_0+F_j(\mlambda)\cdot \mxi'}^2+\delta \abs{\mxi'}^2}}
	\left(\rho v \chi_j\right)(t,\mx)}^{p'}  
	\, d\mlambda \, d\mx \, dt\right)^{1/p'}
	\\ & \quad 
	\leq C \norm{h}_{L^p(\R^{d+1+m})}  
	\norm{\sum\limits_{j=1}^N 
	\abs{\rho v\chi_j}}_{L^{p'}(\R^{d+1+m})}
	\\ & \quad 
	\leq C \norm{h}_{L^p(\R^{d+1+m})}  
	\norm{\rho v}_{L^{p'}(\R^{d+1+m})},
\end{align*} 
where, in the third step, we utilized Lemma \ref{Lp-m}. 
To obtain estimates that are $N$-independent, it is crucial to 
have $\chi_j$ both in front of and under the multiplier 
operator. With this observation, we conclude the 
proof of \eqref{cruc1}.
\end{proof}

Let $A_\psi$ be a Fourier multiplier operator, and let 
$M_b$ represent the multiplication operator by a function $b$. 
Consider the commutator
$$
\cC:=\bigl[\cA_\psi, M_b] 
= \cA_\psi M_b - M_b \cA_\psi.
$$
The subsequent ``non-smooth $b$" version of Tartar's 
commutation lemma has been employed in this paper 
(see \cite[Theorem 3 \& Remark 1]{Antonic:2018aa} 
and the cited sources).

\begin{lemma}\label{lem:commutation}
Consider three numbers $p$, $\bar{p}$, and $q$ such that $p$ 
belongs to the interval $(1,\infty)$, $\bar{p}$ is strictly greater than $p$, 
and $\frac{1}{p} + \frac{1}{\bar{p}} + \frac{1}{q} < 1$. Let $b$ be an 
element of $L^q(\mathbb{R}^D)$, where $D \geq 1$. 
Suppose $\psi$ is a bounded function in $C^\kappa(\mathbb{R}^D \setminus {0})$, 
$\kappa := \left[\frac{D}{2}\right] + 1$, and 
satisfies \eqref{eq:multipliers} (or the more general 
``H{\"o}rmander" condition). Let $\seq{u_n}$ be a sequence that 
converges weakly to $0$ in $L^{\bar{p}}(\mathbb{R}^D)$, and let 
$\varphi$ be an element of $C_c^\infty(\mathbb{R}^D)$. 
Then it follows that 
$$
\cC\bigl(\varphi u_n\bigr)\ton 0 
\quad \text{in $L_{\loc}^p(\R^D)$}.
$$
\end{lemma}

\section*{Acknowledgement}

The current address of Darko Mitrovi\'{c} (DM) 
is the University of Vienna, 
Oskar-Morgenstern-Platz 1, 1090 Vienna, Austria. 
The research of ME is supported by 
the Croatian Science Foundation (UIP-2025-02-1337)
and by the European Union---NextGenerationEU 
through the National Recovery and Resilience Plan 
2021--2026, via an institutional grant of University of Zagreb 
Faculty of Science (IK IA 1.1.3. Impact4Math). 
The research of KHK is supported by the Research Council 
of Norway (351123/NASTRAN), and that 
of DM is partly supported by the 
Austrian Science Fund (P 35508).



\end{document}